\documentclass[journal,onecolumn]{IEEEtran}

\usepackage[margin=1in]{geometry}

\usepackage{amsfonts} 
\usepackage{amssymb}
\usepackage{amsmath}
\usepackage{amsthm}
\usepackage{latexsym}
\usepackage{graphicx}
\usepackage{xypic}
\usepackage{comment}

\usepackage{url,hyperref}
\usepackage[english]{babel}

\hypersetup{citecolor=black,pdfstartview=FitH, pdfpagemode=None}
\def\tcr{\textcolor[rgb]{1,0,0}}

\numberwithin{equation}{section}

\newtheorem{theorem}{Theorem}[section]
\newtheorem{lemma}[theorem]{Lemma}

\theoremstyle{definition}

\newtheorem{remark}[theorem]{Remark}

\renewcommand{\ge}{\geqslant}
\renewcommand{\le}{\leqslant}

\def\R{\mathbb R}

\def\tr{{\rm tr\,}}

\def\vect{\mbox{\rm vec\,}}

\usepackage{algorithm}
\usepackage{algorithmic}
\usepackage{multirow}
\usepackage{amsmath}
\usepackage{xcolor}

\usepackage{bm}

\usepackage{tikz}
\usetikzlibrary{arrows}

\ifCLASSINFOpdf
\else
\fi
\begin{document}
%
\title{Iterative optimal solutions of linear matrix equations for Hyperspectral and Multispectral image fusing}
%
%
%

\author{Frank Uhlig,
        An-Bao Xu* 
\thanks{This work was supported in part by the National Natural Science Foundation of China (11801418) and in part by the grant of the China Scholarship Council (201406130046) }
\thanks{Frank Uhlig is with the Department of Mathematics and Statistics at Auburn University, Auburn, AL 36849-5310, USA \newline
(e-mail: uhligfd@auburn.edu).}
\thanks{An-Bao Xu is at the College of Mathematics and Physics, Wenzhou University, Zhejiang 325035, China \newline
(e-mail: xuanbao@wzu.edu.cn). }}

%
%

\markboth{}%
{Shell \MakeLowercase{\textit{et al.}}: Bare Demo of IEEEtran.cls for IEEE Journals}
%



\maketitle

\begin{abstract}
For a linear matrix  function $f$ in $X \in \R^{m\times n}$  we consider inhomogeneous linear matrix equations $f(X) = E$ for $E \neq 0$  that have or do not have solutions. For such systems we compute optimal norm constrained solutions iteratively using the Conjugate Gradient and Lanczos' methods in combination with  the More-Sorensen optimizer.  We build codes for ten  linear matrix equations, of Sylvester, Lyapunov,  Stein and structured types and their T-versions, that differ only in two five times repeated equation specific code lines.  Numerical experiments with linear matrix  equations are performed that illustrate  universality and  efficiency of our method for dense and small data matrices, as well as for sparse and certain structured input matrices. Specifically we show how to adapt our universal method for sparse inputs and for structured data such as encountered when fusing image data sets via a Sylvester equation algorithm to obtain an image of higher resolution.
\end{abstract}


\IEEEpeerreviewmaketitle

%
\IEEEpeerreviewmaketitle

\section{Introduction}

Linear systems have a long history and near infinitely many uses and applications. The most basic linear vector system is $Ax =b$ for an $ m$ by $n$ matrix $A$ and vectors $x \in \R^n$ and $b \in \R^m$ where $A$ and $b$ are given and $x$ is unknown. Clearly $Ax = b$ is solvable precisely when $b$ lies in the column space of $A$. Otherwise the given system is unsolvable. Yet even then  a best 'near solution' may be useful for applications. And for unsolvable linear vector equations $Ax=b$, one might need to find a vector $x$ that minimizes the residue  $ \| Ax - b\|$ over all $x \in \R^n$, measured in an appropriate norm. This is called the 'least squares problem' for linear matrix vector equations when we use the Euclidean norm.

Generalizing to matrix equations, we call an equation $f(X) = E$ linear if $f$ is linear in the unknown matrix $X$. In this sense, the classical Sylvester equations $AX + XB = E$ with $f(X) = AX+XB$ or $A^TXA + B^TXB = E$ with $g(X) = A^TXA + B^TXB $ are linear matrix equations in the unknown matrix $X$, and so is the commutator equation $AX - XB = O$ for $h(X) = AX - XB$. More specifically, the eigenvector equation $Ax - x \lambda = 0$ for a known eigenvalue $\lambda$ of $A$ is linear, and so forth. Both the continuous Lyapunov equation $A X + X A^T = E$ and its discrete analogue $A X A^T - X = E$ have this same linear form $f(X) = E$ for different linear matrix functions $f$. In this paper we deal with  generalized versions of  inhomogeneous Sylvester, Lyapunov and Stein equations \cite{XM15,SCZ11}, as well as with their transposed or T-versions of the following general form:
 \begin{equation}\label{1.0}
 f(X) = \sum_{k=1}^{k_0} A_k X B_k +  \sum_{j=1}^{j_0} C_j X^T D_j   = E \neq O
 \end{equation}
 where $A_k  \in \R^{p\times m}$,  $B_k \in \R^{n \times q}$,  $C_j  \in \R^{p\times n}$,  $D_j \in \R^{m \times q}$, $E \in \R^{p \times q}$ are given for $k = 1, ..., k_0$, $j = 1, ..,j_0$ and $X \in \R^{m\times n}$ is unknown. Here the row and column sizes $m,n,p,q$ are so that the intended matrix multiplications in (\ref{1.0}) can be performed and either sum can be void. We call  a matrix equation to be of T-type if the second sum (involving $X^T$) is not void, i.e., if $j_0 >0$. 
 
Fusing hyperspectral (HS) and multispectral (MS) images, also known as  multiband image fusion, has recently  drawn special attention in remote sensing \cite{LDABBC15,LMCH18,SBAC15}. Its purpose is to reconstruct a high-spatial and high-spectral multiband image from two degraded and complementary observed images. Based on \cite{WDT} and \cite{WDTD}, this challenging task can be solved by using a Sylvester equation for large sparse matrices with  a certain structure which is just a special case  of  (\ref{1.0}). Here we develop a universal method to solve a multitude of linear matrix equations and then adapt it to solve the sparse structured Sylvester equation for  multiband image fusion problems efficiently.

Here we consider the generalized Sylvester matrix equation $ f(X) = \sum_{k=1}^{k_0} A_k X B_k +  \sum_{j=1}^{j_0} C_j \tcr{X} D_j   = E$ (\ref{1.0}), induced from  multiband image fusion.
 This equation is well studied and has many other applications, see  \cite{BNP,7,15,16,Si14,WDT,ZHNP,LH12}  for example.  
In particular,    we consider its classical Tikhonov regularization: \\
 Find
\begin{equation}\label{1.1}
\mathop {\arg\min }\limits_{X\in \R^{m\times n}} \textstyle{1 \over 2}\|{f(X)-E} \|^2_F +\zeta \| X \|_F,
\end{equation}
where $\zeta >0$ is the regularization parameter.\\
Problem (\ref{1.1}) is equivalent to the following Frobenius norm  'least squares problem' with  norm inequality constraint:
\begin{eqnarray}\label{1.2}
\mathop {\arg\min}\limits_{X \in \R^{m\times n}} \textstyle{1 \over 2}\| {f(X)-E} \|_F^2  
\quad \mbox{subject  to}\quad  \| X \|_F \le \Delta 
\end{eqnarray}
$ \text{for some constant } \Delta  > 0.$
\noindent
The proof of the equivalence of  formulation (\ref{1.2}) and  Problem (\ref{1.1}) is given in Section II.
Finally note that since  $\{ X: \|X\|_F\le \Delta\}$ is  compact, Problem (\ref{1.2})  has at least one global minimum by  Weierstrass' Theorem.\\[-4mm]

The rest of this paper is structured as follows. Section II gives some notations and preliminaries. In Section III we propose and develop a matrix product based iterative method to solve Sylvester type matrix equations that uses the \underline{g}eneralized \underline{L}anczos \underline{t}rust \underline{r}egion algorithm (GLTR)  for solving problem (\ref{1.1}), see \cite{9} and \cite{GLRT}. Our GLTR algorithm is based on the Steihaug-Toint algorithm \cite{9,10}. 
In Section IV we prove general convergence of the method and speed up  the algorithm further. In Section V, applications from the literature and numerical tests will  illustrate the efficiency and accuracy of our algorithm in theoretical and real world applications, both for generalized Sylvester, Lyapunov and Stein equations and their T-versions for dense, sparse and structured sparse matrices, respectively and  all alike.

Previously semi-direct canonical form methods have been used for dense matrix equations problems, while Krylov projection methods are generally preferred for sparse linear matrix equations. The first class of methods is  based on normal form computations of associated matrices and uses Francis' QR algorithm or  SVD computations to form triangular equivalent systems that are then solved for the entries of the unknown solution $X$. See Bartels and Stewart  \cite{BS72} for Sylvester and Kitagawa \cite{K77} or Barraud \cite{B78} for Lyapunov equations. With the advent of multishift Francis QR by Braman, Byers and Mathias \cite{BBM02a, BBM02b}, normal form based methods  could  theoretically be applied for matrix dimensions up to 10,000 by 10,000 and succeed. The Krylov projection approach was developed more recently, see the survey article by Simoncini \cite{Si14} or her earlier paper \cite{Si07} for solving sparse Lyapunov matrix equations and also Dopico \cite{DGKS16} for sparse structured T-Sylvester equations.

Currently  linear matrix equation problems have become huge and structured. Sparse and direct eigen based methods are generally not able to handle such inputs efficiently. Our iterative method relies completely on matrix multiplications and has low overhead and low storage requirements. Our set of  algorithms and their computational codes are an extension and outgrowth of the second named author's two previous papers \cite{XP13,XXP15}, which have dealt with the 1-term Sylvester type matrix equation $AXB = E$. The current paper builds in part on these earlier works and refines the algorithm, as well as extends it  to solve three new classes of linear matrix equations. Moreover, we deal with sparse and structured input matrices as well. In each of our eleven versions for Sylvester-like linear matrix equation problems, only ten lines of code use maximally four matrix multiplications each and our iterations counts stay low. This gives our iterative method a great advantage for dense matrices over canonical form based methods, as well as performing well for general sparse matrices. And moreover, our  codes and method can easily be adapted for structured matrices, see the Subsection C for a computed example from multiband image fusion. Our iterative algorithms work alike for solvable and unsolvable linear matrix equations and do so without any known spectral restrictions on the input matrices that  sparse or structured Krylov methods often encounter, see \cite[Numerical Tests 7.3 through 7.9]{DGKS16} for example.


\section{Notations and Preliminaries}
Throughout this paper, $I$ represents the identity matrix of appropriate dimension, and  $A^T$ and $\left\| A \right\|_F$ denote the transpose and the Frobenius norm of the matrix $A$, respectively.  For  $A=(a_{ij})\in \R^{p\times m}$ and $B=(b_{ij})\in \R^{n\times q}$, $A \otimes B$ denotes the Kronecker product of  $A$ and $B$, that is, $A \otimes B= (a_{ij} B) \in \R^{pn\times mq}$. The inner product in $\R^{k\times \ell}$ is defined by $\langle A,B \rangle =\tr  (B^T A)$ for $A,B\in \R^{k\times \ell}$ and the induced matrix norm then becomes the Frobenius norm.\\[-4mm]

In the algorithms and codes that follow we will use the adjoint function $f^*$ with respect to the inner product $\langle A,B\rangle = \tr(B^TA)$ of a given linear matrix function $f$  in the form (\ref{1.0}). By definition, the adjoint of a linear function $f$ with respect to any inner product $\langle..,..\rangle$ is the function $f^*$ for which $\langle f(x),y\rangle = \langle x, f^*(y)\rangle$ holds for all $x$ and $y$ in their respective domains. Since $f$ in (\ref{1.0}) is linear in each of its terms it suffices  to find the adjoint of a typical Sylvester summand  $h(X) = A_kXB_k$ in (\ref{1.0}) and of its T-Sylvester counterpart $g(X) =  C_jX^TD_j$ individually. Using elementary properties of the matrix trace function, one can easily derive the identity
\begin{eqnarray*}
\langle h(X),Y\rangle = \langle A_kXB_k,Y\rangle = \tr(Y^TA_kXB_k)  = 
\tr(B_kY^TA_kX) = \langle X, A_k^TYB_k^T\rangle = \langle X,h^*(Y)\rangle .
\end{eqnarray*}
Thus the adjoint function of $h$ is  $h^*(Y) = A_k^TYB_k^T$. Likewise for a T-Sylvester  term of the form $g(X) = C_jX^TD_j$ in (\ref{1.0}) we can again use the cyclic property for two or more  factors such as $\tr(AB) = \tr(BA)$ or $\tr (ABC) = \tr(BCA)$ and the symmetric property $\tr(X^TY) = \tr(XY^T)$ for two factor matrix products. Hence if $g(X) = C_jX^TD_j$, then
\begin{eqnarray*}
\langle g(X),Y \rangle = \langle C_jX^TD_j,Y\rangle = \tr(Y^TC_jX^TD_j) =
\tr(X^TD_jY^TC_j) = \tr(C_j^TYD_j^TX) = \langle X,D_jY^TC_j \rangle ,
\end{eqnarray*} 
making $g^*(Y) = D_jY^TC_j$ the adjoint function of a T-Sylvester term $g(X) = C_jX^TD_j$ in (\ref{1.0}).

\subsection{Proof of the equivalence of Problem (\ref{1.2}) and  Problem (\ref{1.1})}
\begin{proof}
According to \cite{6},  to understand the equivalence of  formulation (\ref{1.2}) and  Problem (\ref{1.1}), we observe that if $E \notin \mathcal{R}(f)$ with $\mathcal{R}(f)$ denoting the  range of  $f$, then any solution of Problem (\ref{1.2}) is a minimizer. Therefore the Karush-Kuhn-Tucker conditions for a feasible solution $X$ of Problem (\ref{1.2}) with  corresponding Lagrange multiplier $\lambda$ are (a) $\tcr{f^*(f(X))} +\lambda_* X = f^*(E)$ with $\lambda_* >0$ and (b) $\lambda_* (\|X\|_F -\Delta) = 0$. Furthermore Problem (\ref{1.2}) is a convex quadratic problem and therefore these two conditions are necessary and sufficient. Equivalence with Problem (\ref{1.1}) follows directly, since a solution $X_*$ of Problem (\ref{1.2}) is also a solution of Problem (\ref{1.1}) for $\zeta = \lambda_*$. Conversely, if $X_{\zeta}$ is a solution of Problem (\ref{1.1}) for a given $\zeta$, then $X_{\zeta}$ solves Problem (\ref{1.2}) for $\Delta = \|X_{\zeta}\|_F$. 
\end{proof}

\section{An Iterative Method to Solve Problem (\ref{1.1}) and its Properties}


Written out {explicitly, our model problem (\ref{1.2}) for  generalized Sylvester equations becomes  
\begin{eqnarray}\label{2.1}
\mathop {\arg\min}\limits_{X\in \R^{m\times n}} \frac{1}{2}\langle {f(X),f(X)} \rangle -\langle {f(X),E} \rangle \ 
\mbox{subject  to} \  \left\| X \right\| \le \Delta,
\end{eqnarray}
where from hereon out we continue to write $\|..\|$ instead of subscripting norms by $..._{F}$ as we will always use the Frobenius norm here .

We now describe our iterative method to solve Problem (\ref{2.1})  in basic detail, as designed to solve Problem (\ref{1.1}).

\noindent
\underline{\hspace*{135mm}}\\[1mm]
{\bf Algorithm 3.1:} \ \ Generalized Sylvester Equation;  Basic Version\\[-2mm]
\underline{\hspace*{135mm}}\\[-2mm] 

\noindent
{\bf{Input :}}  Compatibly sized input matrices for $f(..) = E$ and a positive real number $\Delta $. \\[1mm]
{\bf{Initialize :}} \begin{minipage}[t]{135mm} {Start with $X_0 =0 \in \R^{m\times n}$, $Q_{-1} =0 \in \R^{m\times n}$ and a small given tolerance $\varepsilon >0$. \\ [1mm]
Compute
    $R_0 =-(f^*(E))\  \neq 0$, set $t_0 = R_0$, $\gamma_0 =\left\| {R_0 } \right\|$. Then set \\[1mm]
      $P_0 =-R_0$,  $T_{-1}=[]$ (empty),   $k = 0$, Switch = 0 and Done = 0. }\end{minipage}   \\

\hspace*{-3mm} \underline{{\bf{While}}} Switch = 0 and Done = 0 \textbf{do :}  
 \hspace*{5mm}  (First (interior) branch) \\[-3mm]

\ \textbf{1.1 :} \begin{minipage}[t]{135mm}{Compute $Q_k =t_k /\gamma_k $, $\delta _k =\|f(Q_k)\|^2$, 
 $t_{k+1} =f^*( f(Q_k)) -\delta _k Q_k -\gamma _k Q_{k-1} $,\\[1mm]
$\gamma _{k+1} =\left\| {t_{k+1} } \right\|$, and $    T_k = \begin{bmatrix}T_{k-1}& \Gamma _k \\\Gamma _k ^T  & {\delta _k }\end{bmatrix}$,  where $\Gamma _k =(0,\dots ,0,\gamma _k )^T\in \R^k$.} \end{minipage}\\[-1mm]

\ \textbf{1.2 :} \begin{minipage}[t]{135mm} \textbf{If} $f(P_k) \ne 0$,\\[1mm] 
\hspace*{2mm} {Compute} $ \alpha _k =\left\| {R_k } \right\|^2 /\|f(P_k)\|^2$ 
\hspace*{1mm} and $X_{k+1} =X_k +\alpha _k P_k$;\\[1mm]
 \hspace*{2mm} \textbf{If} $\left\| {X_{k+1} } \right\|\le \Delta $, \\[1mm]
 \hspace*{5mm} \begin{minipage}[t]{136mm}{ {Compute} $R_{k+1}=R_k +\alpha _k (f^*( f(P_k) ))$,\\[1mm] 
  \textbf{If} $\left\| {R_{k+1} } \right\| <\varepsilon $, Done = 1, \textbf{End};
{Set}   $\beta _k =\left\|{R_{k+1}}\right\|^2/\left\| {R_k } \right\|^2$, \\
and $P_{k+1} =-R_{k+1} +\beta _k P_k $; }\end{minipage}\\[1mm]
 \hspace*{3mm}\textbf{Else} Switch = 1; \textbf{End} \\[1mm]
  \textbf{Else} Switch = 1; \textbf{End} \end{minipage} \\[1mm]
\hspace*{5mm}\textbf{1.3 :}    {Set} $k =k+1$.\\[-2mm]

\hspace*{-3mm} \underline{\textbf{End While}} \\[-2mm]

\hspace*{-3.4mm} \underline{\textbf{While}} Switch = 1 and Done = 0 \textbf{do:}
 \hspace*{5mm} (Second (boundary) branch) \\[-2mm]

\ \textbf{2.1 :} \begin{minipage}[t]{135mm} {Compute $Q_k =t_k /\gamma_k $, $\delta _k =\|f(Q_k)\|^2$, \  $t_{k+1} =f^* f(Q_k) -\delta _k Q_k -\gamma _k Q_{k-1} $,\\
$\gamma _{k+1} =\left\| {t_{k+1} } \right\|$, and $    T_k = \begin{bmatrix}T_{k-1}& \Gamma _k \\\Gamma _k ^T  & {\delta _k }\end{bmatrix}$, where $\Gamma _k =(0,\dots ,0,\gamma _k )^T\in \R^k$.}\end{minipage}\\

\ \textbf{2.2 :} \begin{minipage}[t]{135mm} {Find the optimal solution $h_k$ of :
\begin{equation}\label{6}
\hspace*{-8mm}\mathop {\min }\limits_{h\in \R^{k+1}} \frac{1}{2}h^TT_k h+\gamma _0 h^Te_1
\ \mbox{subject to}\  \| h \| \le \Delta
\end{equation}
via Algorithm 3.2.} \end{minipage}\\[2mm]
\hspace*{10mm}\begin{minipage}[t]{135mm}{
 \textbf{If} $\gamma _{k+1} \left| {\langle {e_{k+1} ,h_k } \rangle } \right|<\varepsilon$ (for the $k+1^{st}$ unit vector $e_{k+1}$), \\[1mm]
 \hspace*{0mm} {Set} $\tilde {X}_k =(Q_0 ,Q_1 ,\dots ,Q_k )(h_k \otimes I)$, Done = 1;\\
  \textbf{End}}\end{minipage}\\[2mm]
  
\ \textbf{2.3 :} {Set} $k = k+1$.\\[-3mm]

\underline{\textbf{End While}}\\[2mm] 
{\bf{Output :} } {Solution} matrix  $X_*=X_{k+1}$ (from branch 1) or $\tilde {X}_k$ (from branch 2), iterations counter $k$\\[-2mm]
\underline{\hspace*{135mm}}\\[-4mm]

\begin{remark}\label{Remark10}
If $R_0 =-(f^*(E))= 0$ then $X_*=0$ solves  Problem (\ref{1.2}) according to Theorem \ref{Theorem1} of Section IV below.  Therefore we  only consider the case $R_0\ne 0$ in  all  versions of the algorithm that we consider. 
\end{remark}

The basic iteration  of Algorithm 3.1  involves two branches: 
The first uses the Conjugate Gradient (CG) method in step 1.2 and tries to compute the solution of Problem (\ref{1.2}) inside the feasible region $\{X \mid \|X\| < \Delta\}$, see also \cite{GLRT,9}. 

When Problem (\ref{1.2})  cannot be solved in the feasible region via CG,  we solve  Problem (\ref{6}) instead. In this case  the optimal solution lies on the boundary  $\{X : \|X\| = \Delta \}$   according to  Theorems (\ref{Theorem5}) and (\ref{Theorem4}), and it is obtained by the More-Sorensen algorithm \cite{11}. \\[-4mm]

The flow chart of this algorithm is in Figure 1.\\[-2mm] 
\begin{figure*}[t]
\begin{center}
\begin{tikzpicture}[
  font=\sffamily,
  every matrix/.style={ampersand replacement=\&,column sep=0.5cm,row sep=0.35cm},
  source/.style={draw,thick,rounded corners,fill=yellow!20,inner sep=.15cm},
  process/.style={draw,thick, rectangle,fill=blue!20},
  sink/.style={source,fill=green!20},
  datastore/.style={draw,very thick,shape=datastore,inner sep=.3cm},
  dots/.style={gray,scale=2},
  to/.style={->,>=stealth',shorten >=1pt,semithick,font=\sffamily\footnotesize},
  every node/.style={align=center}]

  \matrix{
     \node[source] (Result2) {$X_* = \tilde {X}_k$};
      \& \node[source] (Step1) {Initialize}; 
      \& \node[source] (Result1) {$X_*=X_{k+1}$}; \\

  \node[process] (Criterion2){Stop Criterion\\$\gamma _{k+1} \left| {\langle {e_{k+1} ,h_k } \rangle } \right|<\varepsilon$};
    \&  
    \&   \node[process] (Criterion1){Stop Criterion\\$\left\| {R_{k+1} } \right\| <\varepsilon $}; \\

    \node[sink] (Step4) {Step 2.2: Solve the Subproblem\\\hspace*{10mm} via More-Sorensen};
      \&     \node[source] (Go) {Done=0}; 
      \& \node[sink] (Step3) {Step 1.2: CG method}; \\

    \&  \node[process] (Check) {Switch = 0 \\ and Done = 0};
    \&   \node[process] (Step33){$f(P_k) \ne 0$  \\ and $\left\| {X_{k+1} } \right\| \le \Delta $}; \\

   \node[sink] (Lanczos1) {Step 2.1: Lanczos method};
      \&      \& \node[sink] (Lanczos2) {Step 1.1: Lanczos method}; \\
  };

  \draw[to,thick] (Step1) --(Go);
  \draw[to,thick] (Go) --(Check);
 
  \draw[to,thick] (Check) -- node[midway,above] {Yes} (Lanczos2);
     \draw[to,thick] (Lanczos2) --  (Step33);
      \draw[to,thick] (Step33) --node[midway,above] {No}node[midway,below] {Switch=1 \qquad }   (Go);
      \draw[to,thick] (Step33) --node[midway,left] {Yes}node[midway,right] {Switch=0}   (Step3);
      
   \draw[to,thick] (Step3) --  (Criterion1);
   \draw[to,thick] (Criterion1) -- node[midway,above] {No} (Go);
   \draw[to,thick] (Criterion1) -- node[midway,left] {Yes} node[midway,right] {Done=1}(Result1);

     \draw[to,thick] (Check) -- node[midway,above] {No} (Lanczos1);
     \draw[to,thick] (Lanczos1) --  (Step4);
     \draw[to,thick] (Step4) --  (Criterion2);
       \draw[to,thick] (Criterion2) -- node[midway,above] {No} (Go);
        \draw[to,thick] (Criterion2) -- node[midway,right] {Yes}node[midway,left] {Done=1} (Result2);

\end{tikzpicture}

\end{center}
\hspace*{55mm}Figure 1:  The flow chart of Algorithm 3.1.
\end{figure*}


Now we  detail how to solve Problem (\ref{6}).  This  will complete Algorithm 3.1. Based on Theorem \ref{Theorem5} of Appendix A we solve Problem (\ref{12}) in  Algorithm 3.2 below to get one solution of Problem (\ref{6}). Our method is based  on the work of  \tcr{J. J. More} and D. C. Sorensen in \cite{11} and executed here in  slightly different form.\\[-3mm]

\noindent
\textbf{Algorithm 3.2} \emph{ Constrained optimization according to More-Sorensen} \cite[p. 419]{11}.\\[-3mm] 

\noindent

\hspace*{2mm}\begin{minipage}{75mm}{\begin{enumerate}
\item[Start :] Given a suitable starting value $\lambda _k^0 $ with $T_k +\lambda _k^0 I$ positive definite and $\Delta >0$.\\[-4mm]
\item[Iterate :] For $i=0,1,\dots $until convergence\\[2mm]
\begin{minipage}{65mm}{ \begin{itemize}
 \item[ (a)] Factor $T_k +\lambda _k^i I=LL^T$, where $L$ is lower bidiagonal.\\[-3mm]
    \item [(b)] Solve $LL^Th=-\gamma _0 e_1$ for $h$. \\[-3mm]
\item  [(c)] Solve $Lw=h$ for $w$.\\[-3mm]
 \item  [(d)] Set $\lambda _k^{i+1} =\lambda _k^{i} +\left( {\frac{\left\| h \right\|-\Delta }{\Delta }} \right)\left( {\frac{\left\| h \right\|}{\| {w} \| }} \right)^2$.\\[-0mm]
    \end{itemize}} \end{minipage}
\end{enumerate}}\end{minipage}

\noindent
In  Algorithm 3.2, the initial secular value $\lambda _k^0 $
can be chosen as follows: If  $\left\| {h_k (\lambda _{k-1} )}
\right\| \ge \Delta $, let $\lambda _k^0 =\lambda _{k-1}$; else let $\lambda
_k^0 =0$, where $\lambda _{k-1} $is obtained by the ($k-$1)th iterative steps of
Algorithm 3.1. The stopping criterion is $\left| {\lambda _k
^{i+1}-\lambda _k ^i} \right|\le \varepsilon $, where $\varepsilon $ is a
small chosen tolerance.\\[2mm]
Since each $T_k +\lambda_k I$ is symmetric and bidiagonal, we can implement Algorithm 3.2  in step (a) and likewise in steps (b) and (c) using only the diagonal and subdiagonal vectors of their lower bidiagonal Cholesky factors $L_k$. 


\section{ Main Results and Improvements of Algorithm 3.1}

In this section, we develop solvability conditions for Problem (\ref{1.2}), equivalent to our original Problem (\ref{1.1}), and show that  Problem (\ref{1.2}) can be solved in finitely many  iterations if we disregard rounding errors and if  subproblem (\ref{6}) can be solved. Then we propose a more effective algorithm than  Algorithm 3.1. First let us recall Problem (\ref{2.1}):
\begin{equation*}
\mathop {\min }\limits_{X\in \R^{m\times n}} \frac{1}{2}\langle {f(X),f(X)} \rangle  -\langle {f(X),E} \rangle 
\ \mbox{subject to}  \left\| X \right\| \le \Delta.
\end{equation*}

\begin{theorem}\label{Theorem1} 
\textbf{(Solvability condition)} The matrix $X_\ast $ is a solution of Problem (\ref{1.2}) if and only if $X_\ast $ is feasible, i.e., $\left\| X_\ast \right\| \le \Delta$, and there is a scalar $\lambda ^\ast \ge 0$ such that
\begin{eqnarray}\label{3}
f^*(f(X_\ast)) + \lambda ^\ast X_\ast =f^*(E),\nonumber \
\text{and } \ \lambda ^\ast \cdot (\left\| {X_\ast } \right\|-\Delta)\ =\ 0. 
\end{eqnarray}
\end{theorem}
\noindent

Next we show that  problem (\ref{1.2}) can be solved in
finitely many steps in the absence of rounding errors. When the algorithm does not enter the second branch, Remark \ref{Remark2} of Appendix A tells us that our algorithm has found a solution after finitely many iterations in its first branch.
If the algorithm, however, switches to its second branch we only need to show that the second branch stopping criterion will then be satisfied after finitely many steps. The actual proof of this is as follows.

\begin{theorem}\label{Theorem3}
Suppose that the sequences $\{X_k \}$, $\{R_k \}$ are  generated by
Algorithm 3.1. Then the following equation holds for all $k=0,1,2,\dots $.
\[
f^*(f(X_k)) - f^*(E)=R_k .
\]
\end{theorem}
\begin{proof} We use  induction. For $k=0$ the conclusion holds. Assume that the
conclusion holds for $k-1$. Then
$\hspace*{20mm}f^*(f(X_k)) -f^*(E) =f^*(f(X_{k-1} +\alpha _{k-1} P_{k-1} )) -f^*(E)=$\\
$\hspace*{20mm}=f^*(f(X_{k-1})) -f^*(E) + \alpha _{k-1} f^*(f(P_{k-1})) = R_{k-1} +\alpha _{k-1}  f^*(f(P_{k-1})) = R_k.$ \\[-5mm]
 \end{proof}

\begin{remark}\label{RemarkA}
 Combining Remark \ref{Remark2} and Theorem \ref{Theorem3} we see that 
\[f^*(f(X_{k+1})) -f^*(E)=0.
\]
Hence according to Theorem \ref{Theorem1},  $\lambda ^\ast =0$ and $X_{k+1} $ solves Problem (\ref{1.2}).\\[-4mm]  
\end{remark}

\begin{lemma}\label{Lemma4} \cite[Lemma 4]{XXP15}
Suppose that the sequences $\{Q_k \}$, $\{\gamma _k \}$ \text{and } $\{h_k
\}$ are generated by Algorithm 3.1. Let
\small{
\[
\tilde {X}_k =Q_0 h_k^0 +Q_1 h_k^1 +...+Q_k h_k^k \ \ and \ h_k =(h_k^0 ,h_k^1 ,\dots ,h_k^k )^T.
\] }
Then for all $k=0,1,2,\dots $, there exists a nonnegative number $\lambda _k $ such that
\begin{eqnarray}\label{10}
\ \ f^*(f(\tilde {X}_k)) +\lambda _k \tilde {X}_k -f^*(E) = Q_{k+1} \gamma _{k+1} h_k^k,  \quad   \nonumber
\lambda _k \cdot (\left\| {\tilde {X}_k } \right\|-\Delta )=0,  \ {and}  \ \left\| {\tilde {X}_k } \right\| \le \Delta. \ \
\end{eqnarray}
\end{lemma}

\begin{theorem}\label{Theorem4} \cite[Theorem 4]{XXP15}
Suppose that $\gamma _0 ,\gamma _1 ,\dots ,\gamma _k \ne 0$
and $\gamma _{k+1} =0$.\\
 Then $\tilde {X}_k =Q_0 h_k^0 +Q_1 h_k^1 +...+Q_k
h_k^k$ is the solution of Problem (\ref{1.2}).
\end{theorem}
\begin{proof} 
Since $\gamma _{k+1} =0$ and $\tilde {X}_k =Q_0 h_k^0 +Q_1 h_k^1
+...+Q_k h_k^k $, by Lemma \ref{Lemma4}, we have  
\begin{equation*}
\begin{array}{c}
f^*(f(\tilde {X}_k))+\lambda _k \tilde {X}_k =f^*(E), 
 \lambda _k \cdot  (\| {\tilde {X}_k } \|-\Delta )=0\ \ \text{and}\ \ \| {\tilde {X}_k } \| \le \Delta,
 \end{array}
\end{equation*}
for some $\lambda_k\geq 0$, which imply that $\tilde {X}_k $ solves Problem (\ref{1.2}) according to Theorem \ref{Theorem1}.  
\end{proof}

\begin{remark}\label{Remark3}
Based on Lemma \ref{Lemma2}, the matrices $Q_0 ,Q_1 ,Q_2 ,\dots \in \R^{m\times n} $ are mutually orthogonal.
Hence there exists a positive number $k\le m \cdot n$ such that $Q_k
=0$. Clearly $t_k =\gamma _k Q_k =0$ implies that  $\gamma _k =\sqrt {\langle {t_k
,t_k } \rangle } =0$. Therefore the second
stopping criterion of the algorithm will be satisfied after  finitely many iterations except for rounding errors.
\end{remark}

\begin{remark}\label{Remark4}
\emph{\textbf{(Convergence)}} According to Remarks \ref{Remark2} and \ref{Remark3}, when disregarding  rounding errors a solution is obtained in at most $m \cdot n$ iterations by Algorithm 3.1 when using its first branch exclusively or when using both branches in conjunction.
\end{remark}

According to Theorem \ref{Theorem2}, Algorithm 3.1 can be shortened as  follows.\\[-6mm]

\noindent
\underline{\hspace*{115mm}}\\[1mm]
{\bf Algorithm 4.1:} \ \ Generalized Sylvester Equation; Simplified Version \quad \\[-2mm]
\underline{\hspace*{115mm}}\\[-2mm] 

\noindent
{\bf{Input :}}  Compatibly sized matrices  $A, ...,   E$ and a positive real number $\Delta $. \\[1mm]
{\bf{Initialize :}} \begin{minipage}[t]{115mm} {Start with $X_0 =0$, $Q_{-1} =O_{m,n}$ and a small given tolerance $\varepsilon >0$.\\ 
Compute
    $R_0 =-f^*(E)\  \neq 0$, $t_0 =R_0$, $\gamma _0 =\left\| {R_0 } \right\|$,     $P_0 =-R_0 $,\\
 Set     $T_{-1} =\text{ 'empty'}$ and  $k= 0$, Switch = 0 and Done = 0. }\end{minipage}   
\\

\hspace*{-2mm}\underline{{\bf{While}}} \ Switch = 0 and Done = 0 \textbf{do:} 
\hspace*{5mm} (Interior optimum search)  \\[-3mm]

 \textbf{1.1 :} \ \begin{minipage}[t]{115mm}  \textbf{If}  $f(P_k) \ne 0$, \\[1mm]
\hspace*{1mm}  \begin{minipage}[t]{115mm}{ {Compute} $ \alpha _k =\left\| {R_k } \right\|^2 /\|f(P_k)\|^2$, $Q_k  = (-1)^kR_k /\| R_k \|$,\\[1mm]
 $R_{k+1} =R_k +\alpha _k f^*(f(P_k))$,
  $X_{k+1} = X_k +\alpha _k P_k$,\\[1mm] 
  \textbf{If} $\left\| {X_{k+1} } \right\| \le \Delta $, \\[1mm]
  \hspace*{5mm}\textbf{If} $\left\| {R_{k+1} } \right\| <\varepsilon $, Done=1,    \textbf{End};  \\[1mm]
  \textbf{Else} Switch = 1; \textbf{End}; \\[1mm]
 \small{
 $\beta _k =\left\| {R_{k+1} } \right\|^2/\left\| {R_k } \right\|^2$, $P_{k+1} =-R_{k+1} +\beta _k P_k $,\\[1mm]
$\delta _k = \begin{cases}
\frac 1{\alpha _k}\hspace*{17mm}  if \ k=0\\
\frac 1{\alpha _k}+ \frac {\beta _{k-1} }{\alpha _{k-1}}  \quad if\ k>0
\end{cases}$,
$\gamma _{k+1} =\sqrt {\beta _k } /\alpha _k,$\\[1mm]
$T_k = \begin{bmatrix}T_{k-1}& \Gamma _k \\\Gamma _k ^T  & {\delta _k }\end{bmatrix}$ for
 $\Gamma _k =(0,\dots ,0,\gamma _k )^T\in \R^k$.}
  }\end{minipage}\\[1mm]
  \textbf{Else} Switch = 1;\\[1mm] \textbf{End}; \end{minipage} \\[0mm]
  
\ \textbf{1.2 :}     {Set} $k =k+1$.\\[-2mm]

\hspace*{-2mm}\underline{\textbf{End While}} \\[-3mm]

\hspace*{-2mm}\underline{\textbf{While}} Switch = 1 and Done = 0 \textbf{do:} 
\hspace*{5mm} (boundary optimum search) \\[2mm]
\hspace*{5mm}\textbf{2.1 :} \begin{minipage}[t]{115mm} {Compute $Q_k =t_k /\gamma_k $ ($\mbox{set first } Q_k =(-1)^kR_k /\left\| {R_k } \right\|$), \\[1mm]
$\delta _k =\|f(Q_k)\|^2$,  
 $t_{k+1} =f^*(f(Q_k)) - \delta _k Q_k -\gamma _k Q_{k-1} $,\\[1mm]
$\gamma _{k+1} =\left\| {t_{k+1} } \right\|$,  $    T_k = \begin{bmatrix}T_{k-1}& \Gamma _k \\\Gamma _k ^T  & {\delta _k }\end{bmatrix}$ for $\Gamma _k =(0,\dots ,0,\gamma _k )^T\in \R^k$.}\end{minipage} \\[-2mm]

\ \textbf{2.2 :} \begin{minipage}[t]{115mm} {Use Algorithm 2.2 to compute the solution $h_k $ of Problem (\ref{12}),\\[1mm]
 \textbf{If} $\gamma _{k+1} \left| {\langle {e_{k+1} ,h_k } \rangle } \right|<\varepsilon\ $ (here  $e_{k+1} \in \R^{k+1}$ is the $k+1^{st}$ unit vector)\\[1mm]
 \hspace*{0mm} {Set} $\tilde {X}_k =(Q_0 ,Q_1 ,\dots ,Q_k )(h_k \otimes I)$, Done = 1;\\[1mm]
\textbf{End};}\end{minipage}\\[2mm]

\ \textbf{2.3 :}     {Set} $k = k+1$.\\[-2mm]

\underline{\textbf{End While}}\\[2mm] 
{\bf{Output :} } {Solution} matrix  $X_*=X_{k+1}$ or $\tilde {X}_k$, iterations counter $k$.\\[-2mm]
\underline{\hspace*{115mm}}\\[-3mm]

\noindent
The flow chart for the simpler and faster version in Algorithm 4.1 is in Figure 2.\\[-3mm] 

\begin{figure*}[t]
\begin{center}
\begin{tikzpicture}[
  font=\sffamily,
  every matrix/.style={ampersand replacement=\&,column sep=0.5cm,row sep=0.5cm},
  source/.style={draw,thick,rounded corners,fill=yellow!20,inner sep=.15cm},
  process/.style={draw,thick, rectangle,fill=blue!20},
  sink/.style={source,fill=green!20},
  datastore/.style={draw,very thick,shape=datastore,inner sep=.4cm},
  dots/.style={gray,scale=2},
  to/.style={->,>=stealth',shorten >=1pt,semithick,font=\sffamily\footnotesize},
  every node/.style={align=center}]

  \matrix{
     \node[source] (Result2) {$X_*=\tilde {X}_k$ };
      \&  \node[source] (Step1) {Initialize};
      \& \node[source] (Result1) {$X_*=X_{k+1}$}; \\

  \node[process] (Criterion2){Stop Criterion\\$\gamma _{k+1} \left| {\langle {e_{k+1} ,h_k } \rangle } \right|<\varepsilon$};
    \&  
    \&   \node[process] (Criterion1){Stop Criterion\\$\left\| {R_{k+1} } \right\|<\varepsilon $}; \\

    \node[sink] (Step4) {Step 2.2: Solve the Subproblem\\ \hspace*{10mm} via More-Sorensen};
      \&     \node[source] (Go) {Done=0}; 
      \& \node[sink] (Step3) {Step 1.1: CG method \\  (Invisible Lanczos method)}; \\

       \node[sink] (Lanczos1) {Step 2.1: Lanczos method};
    \&  \node[process] (Check) {Switch = 0 \\ and Done = 0};
    \&   \node[process] (Step33){$f(P_k) \ne 0$  \\ and $\left\| {X_{k+1} } \right\|\le \Delta $}; \\
  };

  \draw[to,thick] (Step1) --(Go);
  \draw[to,thick] (Go) --(Check);
 
  \draw[to,thick] (Check) -- node[midway,above] {Yes} (Step33);
      \draw[to,thick] (Step33) --node[midway,above] {No}node[midway,below] {Switch=1 \qquad}   (Go);
      \draw[to,thick] (Step33) --node[midway,left] {Yes}node[midway,right] {Switch=0}   (Step3);
      
   \draw[to,thick] (Step3) --  (Criterion1);
   \draw[to,thick] (Criterion1) -- node[midway,above] {No} (Go);
   \draw[to,thick] (Criterion1) -- node[midway,left] {Yes} node[midway,right] {Done=1}(Result1);

     \draw[to,thick] (Check) -- node[midway,above] {No} (Lanczos1);
     \draw[to,thick] (Lanczos1) --  (Step4);
     \draw[to,thick] (Step4) --  (Criterion2);
       \draw[to,thick] (Criterion2) -- node[midway,above] {No} (Go);
        \draw[to,thick] (Criterion2) -- node[midway,right] {Yes}node[midway,left] {Done=1} (Result2);

\end{tikzpicture}
\end{center}
\hspace*{55mm}Figure 2:  The flow chart of Algorithm 4.1.
\end{figure*}

\section{ Applications, Numerical Tests and Comparisons }

We have adapted our algorithm to solve nine Sylvester and T-Sylvester type inhomogeneous linear matrix equations $f(X) = E \neq 0$ that fall into 4 different classes in Table 1. Note that we also adapt a fast $f$ and $f^*$ matrix product implementation code in a model that solves structured sparse Sylvester equations  for multiband image fusion via the equation  $\mathbf {\tcr{AX + XD= E}}$ in {\tt fastmult\_SGLTR\_3i\_ADE.m}.
 \\

\begin{table*}[t]\small
\hspace*{29mm}\begin{tabular}{lcr}
\multicolumn{3}{c}{1-term Sylvester like equations :}\\[1mm]
{General 1-term equation}&{\bf AXB = E}&SGLTR\_1t\_ABE.m\\[-2mm]
\multicolumn{3}{c}{\underline{\hspace*{40mm}}}\\[1mm]
\multicolumn{3}{c}{2-term Sylvester equations :}\\[1mm]
{Classical equation}&{\bf AX + XD = E}&SGLTR\_3i\_ADE.m\\[1mm]
Generalized equation&{\bf AXB  + CXD = E}&SGLTR\_5i\_ABCDE.m\\[1mm]
Stein equation&{\bf AXB + X = E}&St\_SGLTR\_ABE.m\\[-2mm]
\multicolumn{3}{c}{\underline{\hspace*{40mm}}}\\[1mm]
\multicolumn{3}{c}{T-Sylvester equations :}\\[1mm]
{Classical T-equation}&{$\mathbf {AX + X^TD= E}$}&T\_SGLTR\_3i\_ADE.m\\[1mm]
Generalized T-equation&{$\mathbf {AXB  + CX^TD = E}$}&T\_SGLTR\_5i\_ABCDE.m\\[1mm]
Stein T-equation&{$\mathbf {AXB + X^T = E}$}&TSt\_SGLTR\_3i\_ABE.m\\[-2mm]
\multicolumn{3}{c}{\underline{\hspace*{40mm}}}\\[1mm]
\multicolumn{3}{c}{Lyapunov equations :}\\[1mm]
{Discrete version}&{$\mathbf {AXA^T - X= E}$}&dLyap\_SGLTR\_AE.m\\[1mm]
Continuous version&{$\mathbf {AX  + XA^T = E}$}&cLyap\_SGLTR\_AE.m\\[-2mm]
\multicolumn{3}{c}{\underline{\hspace*{40mm}}}\\[1mm]
\end{tabular}\\[-0mm]
\hspace*{32mm}Table 1:  Nine Sylvester and T-Sylvester type matrix equations in four classes.
\end{table*}


\noindent
The MATLAB m-files in the above list differ in just ten entry lines where the respective equation defining linear functions $f(X)$ and their  adjoints $f^*(...)$ have been adjusted for each of the ten different linear matrix function $f$. All ten program codes have been tested and they are available on-line at \cite{UXSE2016}. Several detailed numerical examples follow below.

\subsection{Small Random Coefficient Matrix Case.}
Our first test uses  small and simple random entry matrices to show that Algorithm 4.1 finds the unique norm bounded solution $X$ of the Sylvester equation $AXB + CXD = E \neq 0$ precisely except for rounding errors. For a given set of size compatible random entry matrices $A = 2 \cdot randn(7,5), B = 4\cdot randn(5,6), C = -3 \cdot rand(7,5)$ and $D = 2 \cdot randn(5,6)$ we construct a random integer entry matrix $X = floor(10 \cdot randn(5,5))$ and compute $E = A \cdot X \cdot B + C \cdot X \cdot D \ne 0$. Then we call our Matlab function {\tt SGLTR\_5i\_ABCDE.m} with inputs $A, B, C, D, E, \Delta$, and $err$ and compute the solution $X_* \approx X$ as follows where we vary $\Delta$  from below $\|X\| = 43.9659$ to exceeding $\|X\|$.\\[-4mm]

\begin{figure}
\begin{center}
\hspace*{-4mm}\includegraphics[width=95mm,height=95mm]{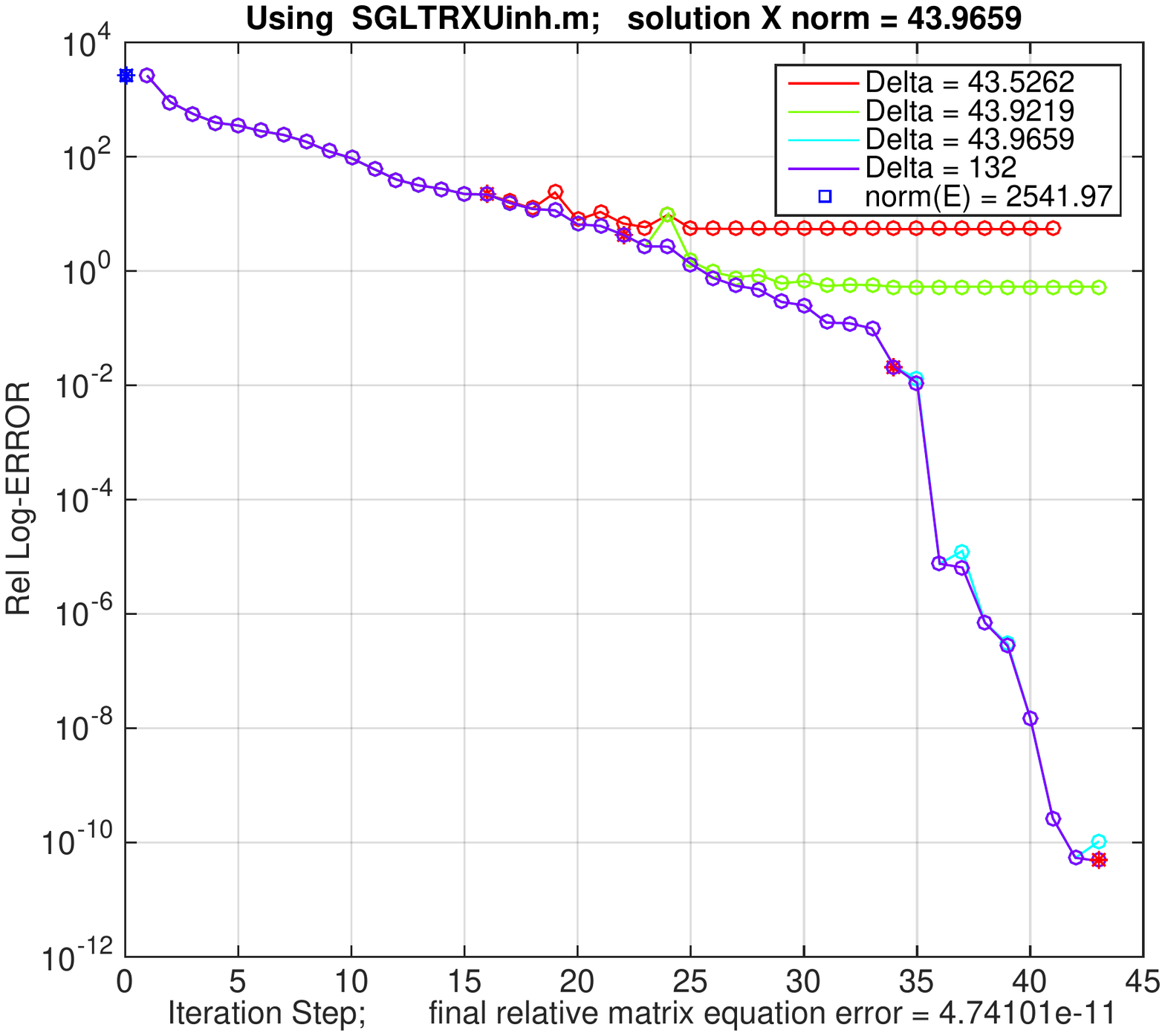}\\[-36mm]
Figure 3
\end{center}
\end{figure}

\vspace*{2mm}
For all  $\Delta \ge \|X\|$ the prescribed solution $X$ is retrieved with small inaccuracies  in the last 2 or 3 digits of Matlab's sixteen and for $\Delta = 0.99 \cdot \|X\|$ 
 and $0.999 \cdot \|X\|$ the all integer entries of the solution matrix $X$  become increasingly recognizable in the computed $X_*$. Note further that in this example the algorithm takes between 30 and 43 iterations which exceed the theoretical convergence bound of $ m \cdot n = 5 \cdot 5$.
 \\[-5mm]

Our second example expands on the first. Here we perturb the fixed right hand side matrix $E$ of the Sylvester equation for the same random entry matrices $A$ through $D$ and now try to solve $A\cdot X \cdot B + C \cdot X \cdot D = E_{pp}$ where $E_{pp} = E + E_p$ with $E$ as before. Here $E_p$ is a small perturbation random entry matrix with $\|E_p\| = \|X\|/10$. In this example the norm of its solution $X_{*p}$ must differ from the earlier solution $X_*$ of the previous unperturbed example. By  construction, the perturbed equation is unsolvable. Since  $A\cdot X_* \cdot B + C \cdot X_* \cdot D - E \approx 0$, the residual error $\|A\cdot X_{*p} \cdot B + C \cdot X_{*p} \cdot D - E_{pp}\|$ of the perturbed Sylvester equation can  at most equal the right hand side perturbation of size $\|E_p\|$. This is borne out in the following graph in which the  horizontal line is drawn at $\|E_p\| = 4.3966$ and the final relative matrix error of the optimal solution $X_{*p}$ of the perturbed system with  $\|X_{*p} \| \le \Delta$ has the size 2.8259 which is well below $\|E_p\|$ once $ \Delta $ is chosen to exceed $\|X\|$, see the annotations of Figure 4.\\[-4mm]

\begin{figure}
\begin{center}\vspace*{-12mm}
\hspace*{-4mm}\includegraphics[width=90mm,height=95mm]{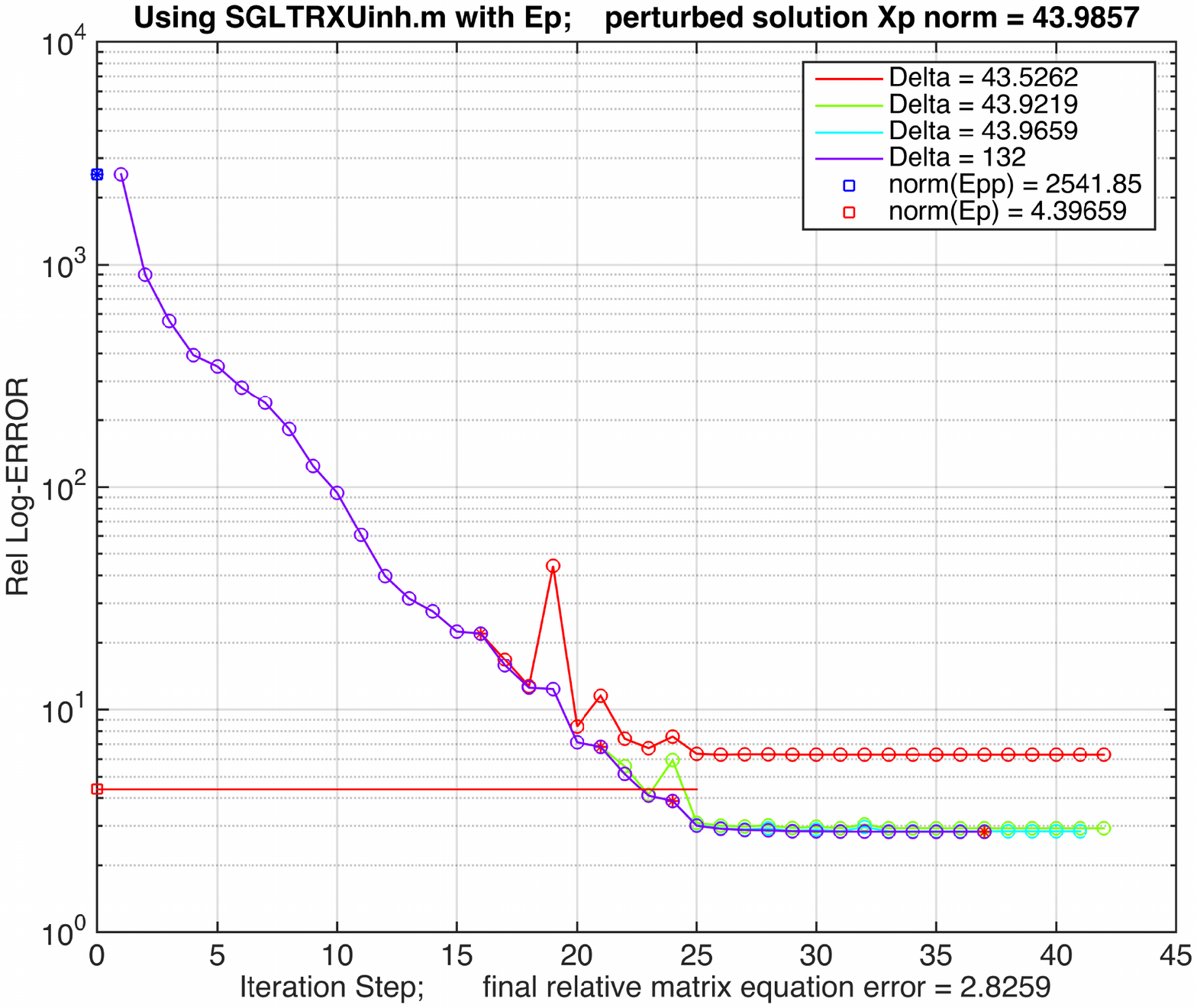}\\[-19mm]
Figure 4\\[-9mm]
\end{center}
\end{figure}
Note that the relative residual matrix equation errors do not decrease monotonically in general. But the norms $\|X_i\|$ of the iterates $X_i$ increase monotonically in practice as they do in theory, see Lemma \ref{Lemma5} and Figure 5.
\begin{figure}
\begin{center}\vspace*{-15mm}
\hspace*{-4mm}\includegraphics[width=90mm,height=100mm]{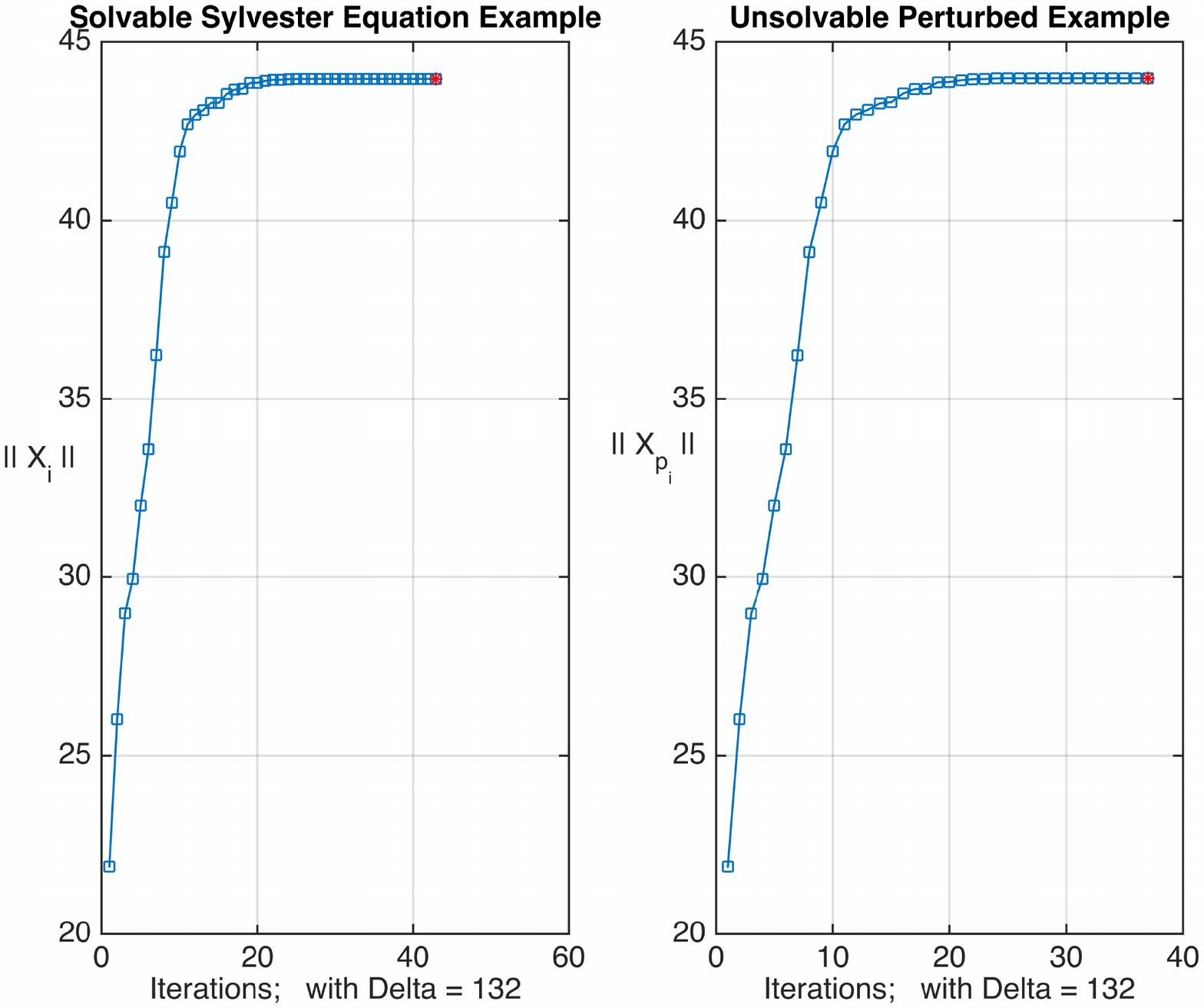}\\[-19mm]
Figure 5 \\[-9mm]
\end{center}
\end{figure}

In this example our algorithm takes 30 to 35 iterations, again exceeding the theoretical maximal iterations bound of $m \cdot n = 5 \cdot 5 = 25$ slightly. We have also investigated the effect of not starting  with $X_0 = 0$. What if we started with a matrix $X_0 \neq 0$ of norm $\Delta / 10$ or even larger? This extended the number of iterations by  30 to 40 \% and gave no better results at all, especially when $X_0$ was chosen with $\|X_0\| > \|X_{opt}\|$ for the given input matrices $A$ through $E$.\\

The final example in this subsection involves three matrices $A$, $B$, and $E$, each of size 28 by 28, for which the classical Sylvester equation $ A X + X B = E$ cannot be solved with any $\Delta$ according to the well know theory. In our chosen random entries example,  the global optimal solution $X$ has a Frobenius norm of approximately 2600. We vary $\Delta$ from  29 through 5800 and record whether the norm restricted optimal solution $X$ was computed on the boundary sphere $\{X  \mid \| X \| = \Delta\}$ or in the interior $ \{ X \mid \{\| X \| < \Delta\}$ and depict the relative matrix equation error  $\|AX + XB - E\|/\|X\|$ in Figure 6. The blue dots in the plot indicate for which $\Delta$ the optimal norm constrained solution was computed on the $\Delta$-sphere, while the red + signs indicate that for these $\Delta$ values, the optimal solution was found inside the $\Delta$-sphere.
\begin{figure}
\begin{center}\vspace*{-26mm}
\hspace*{-4mm}\includegraphics[width=90mm,height=100mm]{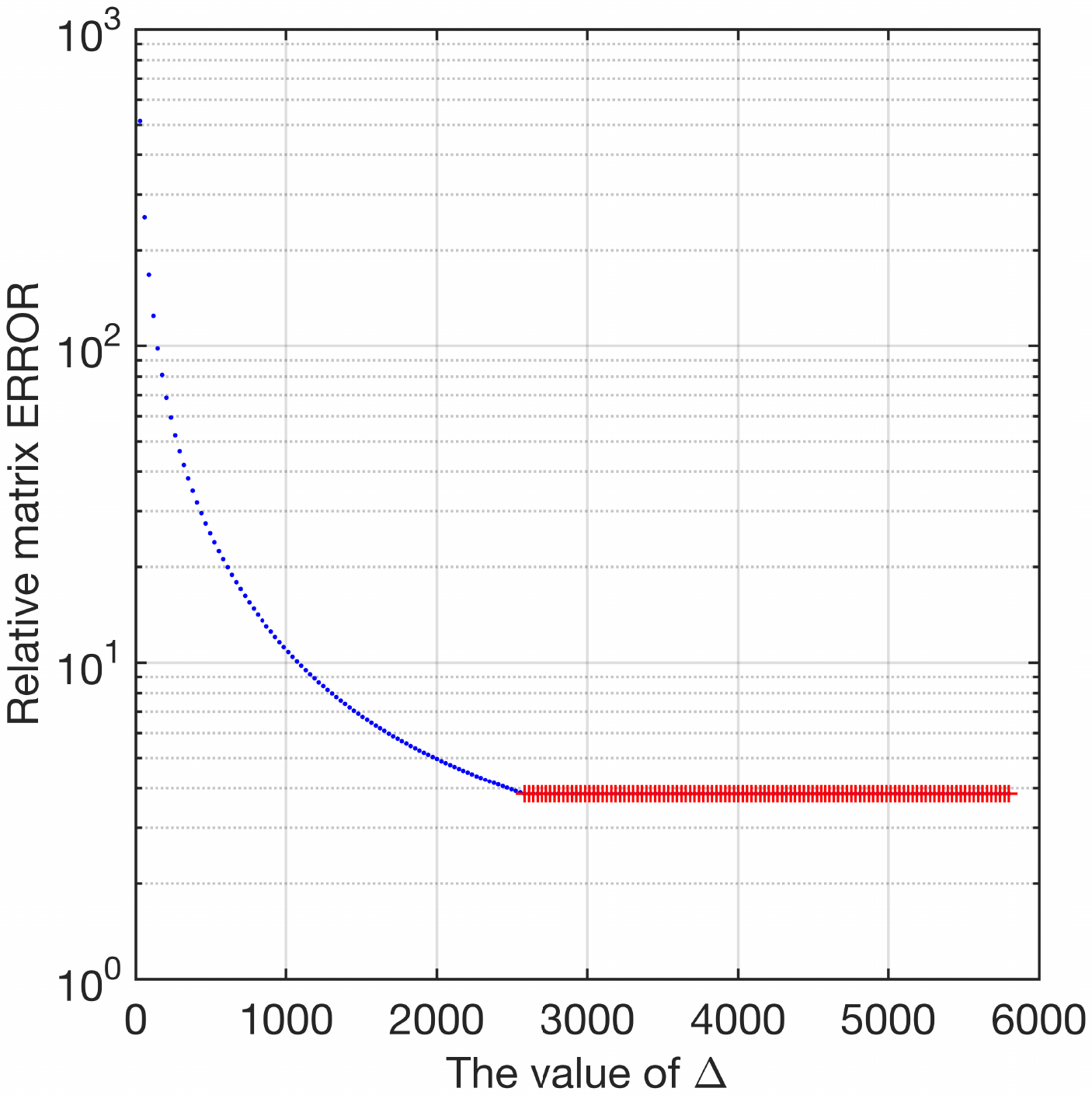}\\[-21mm]
Figure 6\\[-9mm]
\end{center}
\end{figure}
For $\Delta < 2600$ the maximal of interior branch iterations was achieved just below $\Delta =2600$ with 279 interior branch iterations and maximally 39 additional boundary branch iteration steps. For $\Delta > 2600$ the algorithm only used the interior branch and 1,902 iterations for every $\Delta > 2600$. For  smaller $\Delta \leq 1100$ our algorithm used maximally just 4 interior steps and maximally 18 boundary steps. This data is representative for many similarly sized examples with  a relative break-off tolerance of $10^{-10}$.



\subsection{ Sparse Coefficient Matrix Case.}

Here we compare our method with MATLAB's built in {\tt sylvester.m} function for solving $C_1 X + X C_2 = C_3$ for data from \cite[formula (7)]{WDT} and with a simple eigenvalue eigenvector approach suggested by Dopico \cite{DPC}. MATLAB's Sylvester equation solver is based on canonical forms and eigenspace computations as well as blocking methods. It is built on the work of Bartels and Stewart \cite{BS72} and further extentions of this method by Jonsson and K\aa gstr\"om \cite{JK2002i,JK2002ii}. 
 In the computed example below $C_1$ is 4 by 4, $C_2$ is 6400 by 6400 and sparse, and $C_3$, as well as $X$  are 4 by 6400 data matrices. 
 Since $C_1$ is small (4 by 4), an appealing solution method \cite{DPC} might be to diagonalize $C_1 = V D V^{-1}$ and solve 
 $$ VDV^{-1}X + XC_2=C_3$$
 in its  equivalent form
 $$ D(V^{-1}X) + (V^{-1}X) C_2 = V^{-1}C_3 = \tilde C_3.
 $$
 We do this one row of $\tilde X = V^{-1}X$ at a time by Gaussian elimination for the linear system 
 $$\tilde X(j,:) (d_j I + C_2) = \tilde C_3(j,:)$$ and $j = 1,..., 4$. Then $X = V \tilde X$ solves $C_1 X + X C_2 = C_3$.\\[4mm]
 Below we include time and accuracy data with $C_2$ (of size $6400 \times 6400$) in full and sparse modes for these methods that were performed on the same platform with $err = 10^{-14}$, see Table 2. \\
 \begin{table}[t]\footnotesize
  \begin{center} 
\begin{tabular}{l||c|c|c}
 & run time & speed up& relative error\\
& average& factor & ${\|C_1 X + X C_2 - C_3\| \over \|X\|}$\\[2mm] \hline &&\\[-4mm]  \hline \\[-1mm]
Using MATLAB's& & nominal\\
{\tt sylvester.m} :& 55.4 sec& 1& 1.7606e-15\\[1mm] 
Using eig and Gauss :&&\\
with full $C_2$ \cite{DPC}:& 12.96& 4.3& 1.82e-15\\
with sparse $C_2$ in each&&\\
linear system \cite{DPC}:& 2.35 & 23.6 &1.26e-15\\[1mm]
\hline &&\\[-3mm]
With our CG  method :&& \\
for data matrices &&\\
in full matrix mode:& 7.36 sec& 7.5&1.2287e-15\\
in sparse matrix mode:& 0.685 sec&80.9& 1.2314e-15\\
\end{tabular}\\[3mm]
Table 2
\end{center} 
 \end{table}
The solutions $X_{sparse}$ (with $C_2$ in sparse matrix mode) and $X_{full}$ ($C_2$ in full mode)  computed via CG plus Lanzcos or Gauss,  $Y$ from {\tt sylvester.m}  (with $C_2$ necessarily in full mode) and the  solution from \cite{WDT} differ very slightly : 
 $\|X_{sparse} - X_{full}\| = 3.9984e-14, \ \|X_{sparse}-Y\| = 1.8783e-13$, 
   $\|X_{full}-Y\| = 1.8721e-13$, and $\|X_{sparse}\| = \|X_{full}\| = \|Y\| = 1.7970e+01$.\\[2mm]
On a platform in Toulouse, Qi Wei compared our CG algorithm to the algorithm developed in \cite{WDT} that was designed for structured huge input data but not for unstructured sparse problems such as ours is. There our method took 3.3 seconds while the method of \cite{WDT}  took 27 seconds. More comparisons with structured data inputs of our and the  method of \cite{WDT} follow in Section IV.3.\\[-4mm]
 
We repeat that our CG plus Lanzcos method is iterative and uses matrix multiplications throughout which work for both sparse and full matrices in MATLAB without requiring any changes in the code while {\tt sylvester.m}  - by using canonical forms and blocking techniques - can only work when all data matrices $C_i$ are in full MATLAB mode and the simple eigen based method of \cite{DPC} requires one of $C_1$ or $C_2$ to be a very small matrix.

\subsection{Coefficient Matrix with Kronecker Structure.}

This example deals with another image fusion problem with a set of much larger image data matrices one of which has Kronecker structure. It compares the recent work of Qi Wei,  Nicolas Dobigeon, and Jean-Yves Tourneret \cite{WDT} with our algorithm when used for structured left hand side data matrices $A$ through $D$. \\[1mm]
 Here we use  image data that was acquired over Moffett Field in California in 1994 by  JPL and NASA airborne visible and infrared imaging spectrometers (AVIRIS) \cite{GR}. The original image set has high-spatial and high-spectral resolution. It contains 224 images $X_i$, each  of size $390 \times 180$ for $i=1, \cdots, 224$. The original data is stored as a third-order tensor $\mathcal{X}$ of dimensions $390 \times 180 \times 224$. Its matrix version is $X=(\vect \mathcal{X})^T \in \mathbb{R}^{224 \times 70200}$ where we  define $\vect \mathcal{X} := [\vect X_1, \vect X_2,\cdots, \vect X_{224}]$. Each  row of $X$ contains one image $X_i$. In practice, the high spatial and high-spectral image data $X$ is unknown. 

One aim of image  fusion is to approximate the unknown high-spatial high-spectral data from known high-spatial low-spectral multispectral   (MS) data (or high spatial resolution panchromatic (PAN) data) with low-spatial high-spectral hyperspectral (HS) data. These known complementary image data sets result from linear spectral and spatial degradations of the full resolution image data $X$, according to the well-regarded  model developed in \cite{WDTD}, \cite{WBD} and \cite{WDT}
\begin{equation}\label{WDT1} 
Y_M = LX+N_M, \    Y_H =XBS+N_H,  \ \text{ where }
\end{equation}
\begin{itemize}
\vspace*{1mm}
  \item $X  \in \mathbb{R}^{224 \times 70200}$ is the full resolution target image data with 224 bands (or rows) and 70200 pixels (or columns). A composite RGB image can be formed by selecting the red (or 28th), green (or 19th), blue (or 11th) bands of the target image data. This is shown in Figure 7 (a) as taken from \cite{WBD} and \cite{WDT}.\\
  \item $Y_M \in \mathbb{R}^{1 \times 70200}$ and $Y_H \in \mathbb{R}^{224 \times 2808}$ are the given MS and HS image data matrices, respectively. If the band number of MS is fewer than the subspace dimension we set below, the MS data generates into a PAN data. The one band MS image data $Y_M$, as PAN image,  with the composite color image of the HS image data $Y_H$  are shown in  Figures 7 (b) and 7 (c).\\
  \item $L \in \mathbb{R}^{1 \times 70200}$ is the $apriori$ known spectral degradation, which depends on the spectral response of the MS sensor, see  \cite{WDT}  and \cite{YMI} again.\\
  \item $B \in \mathbb{R}^{70200 \times 70200}$ is a cyclic convolution operator acting on each band. Specifically, $B$ has the factored form $B=FDF^H$, where $D \in \mathbb{R}^{70200 \times 70200}$ is a diagonal matrix from \cite{WDTD} and \cite{WDT}, and $F$ and $F^H$ are the discrete Fourier transforms (DFT) and the inverse DFT transform, respectively ($FF^H = F^H F = I_{70200}$). \\
  
  \item $S \in \mathbb{R}^{70200 \times 2808}$ is a downsampling matrix (with downsampling factor denoted by $25=5 \times 5$) acting on each band $(\vect X_i)^TB \text{ for } i=1,\cdots, 224$ as introduced in  \cite{WDT}; Moreover, the downsampling matrix satisfies the property $S^H S = I_{2808}$ and the matrix $S S^H = I_{70200}$ is idempotent, i.e., $(S^H S)(S S^H) = (S^H S)^2 = S^H S$. $S^H S$. Based on the Lemma 1 of \cite{WDTD}, the following result are obtained 
  \begin{equation}\label{4.2}
  F^H S S^H F = {1 \over 5} (1_5 \otimes I_{2808})(1^T_5 \otimes I_{2808}),
  \end{equation}
  where $1_5 \in \mathbb{R}^5$  is a vector of ones. \\
  
  \item $N_M$ and $N_H$ are the MS and HS noise matrices, respectively. The noise matrices are assumed  distributed according to the following matrix normal distributions \cite{WDT}.\\
 \begin{equation*}
N_M \sim \mathcal{MN}(0, \Lambda_M, I), \   \  N_H \sim \mathcal{MN}(0, \Lambda_H, I).
\end{equation*}
In this specific image fusion, following \cite{WDT} we  ignore the noise terms $N_M$ and $N_H$ by setting both $\Lambda_M$ and $\Lambda_H$ equal to the identity matrix.  \\
\end{itemize}

The solution $X=[\vect X_1, \vect X_2,\cdots, \vect X_{224}]^T$ of this problem does not have full rank because each  band (row) $(\vect X_i)^T$  lies in a subspace whose dimension (set 10) is much smaller than the number of bands 224. In other words,  $X=HU$ where $H$ is a full column rank matrix and $U$ is the projection of $X$ onto the subspace spanned by the column of $H$. In this image simulation, the matrix $H$ is determined from a principal component analysis (PCA) of the HS data $Y_H$  as explained in \cite{WBD}.  

It is clear how to formulate a  Sylvester matrix equation from the linear model (\ref{WDT1}) directly according to the discussions of  \cite{WDTD} and  \cite{WDT} for reconstructing the target image in absence of  regularization. Namely:
\begin{equation}\label{WDT7}
C_1 U + U C_2=C_3 
\end{equation}

\vspace*{2mm}
\noindent
where $C_1=(H^H \Lambda_H^{-1} H)^{-1} ((LH)^H \Lambda_M^{-1} LH) \in \mathbb{R}^{10 \times 10}$, $C_2 = BS(BS)^H \in \mathbb{R}^{70200 \times 70200}$ and\\[1mm]
$C_3 = (H^H \Lambda_H^{-1} H)^{-1} (H^H \Lambda_H^{-1} Y_H(BS)^H + (LH)^H \Lambda_M^{-1} Y_M) \in\mathbb{R}^{10 \times 70200}$.\\[1mm]

After solving this Sylvester matrix equation without regularization for $U_*$,  the desired fused image is $X_*=HU_*$. The key issues here is how to solve the Sylvester equation (\ref{WDT7}). Its main difficulty is the huge size of $C_2 = BS(BS)^H \in \mathbb{R}^{70200 \times 70200}$. The second Sylvester term   of (\ref{WDT7}) contains $C_2$ as a factor. $C_2$ has around $5 \cdot 10^{10}$ entries and thus is too huge to construct, to compute with directly or to store explicitly. And besides, $C_1$ is not small as was the case  in Section IV.2 .\\

Fortunately, 
$C_2= BS(BS)^H = F D (F^H S S^H F) D^H F^H$ has a specific structure as the product of the Kronecker structure sparse matrix $F^H S S^H F$, the DFT matrix $F$ and the diagonal matrices $D$ according to formula (\ref{4.2}).  Note that our  algorithm  does not destroy the Kronecker structure and sparsity of $F^H S S^H F$,  nor does it destroy diagonal matrices such as  $D$. Thus our method can take advantage of these  properties in every iteration step. We have deposited our fast implementation of the matrix product $U(C_2)$ for this specific example in \cite[fastmult\_SGLTR\_3i\_ABE.m]{UXSE2016}. There  the $f$ and $f^*$ fast matrix multiplication subroutines {\tt BluSparse} and {\tt TBluSparse} are attached for this specific problem. These codes can easily be adapted to other structured Sylvester matrix equation problems.

A further advantage of our method  is that it finds the regularization solution of equation (\ref{WDT7}) when we choose a suitable $\Delta$. Figure 7 (d) shows the composite color image fusion result obtained by  Algorithm 4.1 using the   problem specific fast codes for the matrix products that involve $...C_2$ and $...C_2^T$ that appear in the matrix function $f(...)$ and its adjoint function $f^*(...)$. In our runs we have experimented with choosing  $640 \leq \Delta \leq 2000$ and have obtained near identical optimal results, all in nearly the same CPU run times and with iteration counts differing by at most two when $\Delta$ was changed.\\[-4mm] 
\begin{figure*}[t]
\begin{center}
  \includegraphics[scale=0.55]{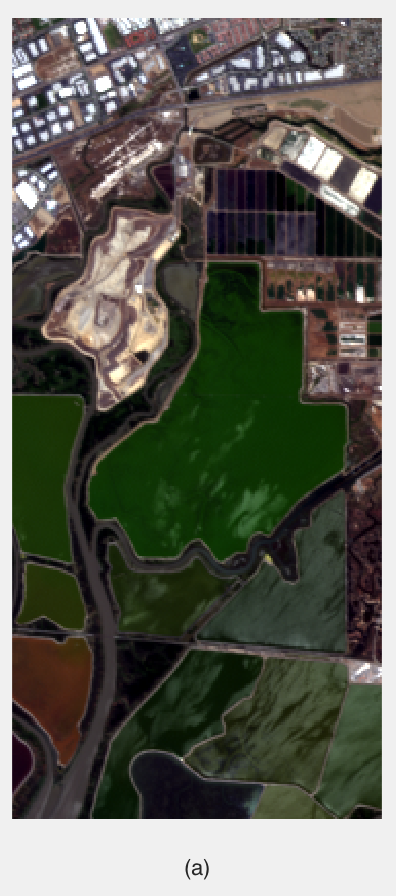}
 \includegraphics[scale=0.55]{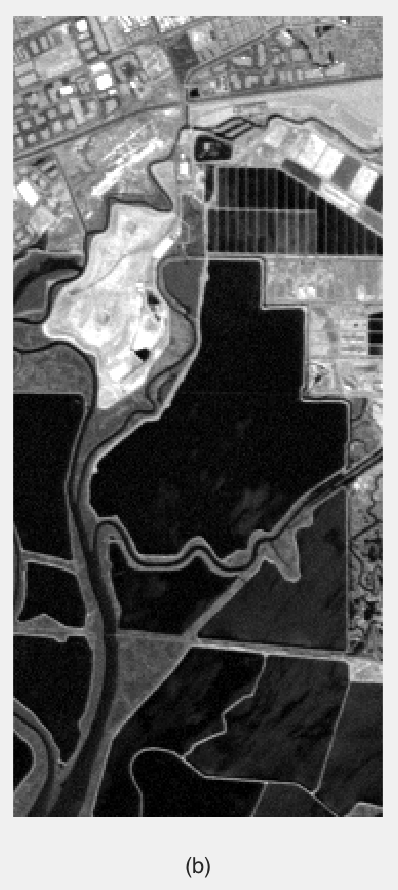}
  \includegraphics[scale=0.55]{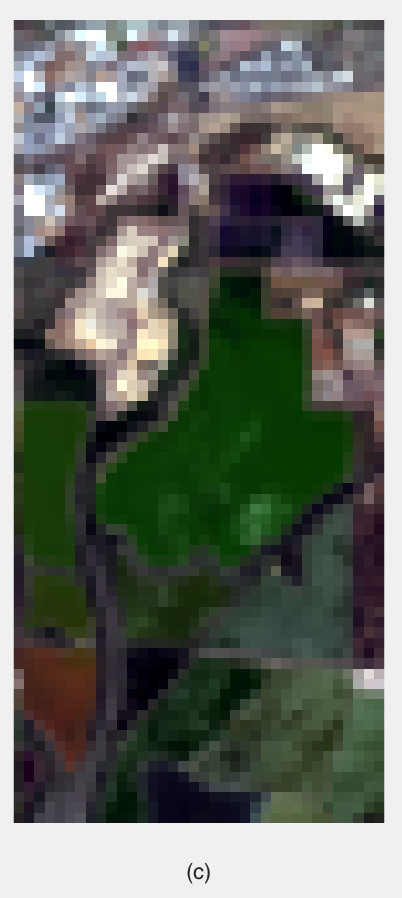}
   \includegraphics[scale=0.55]{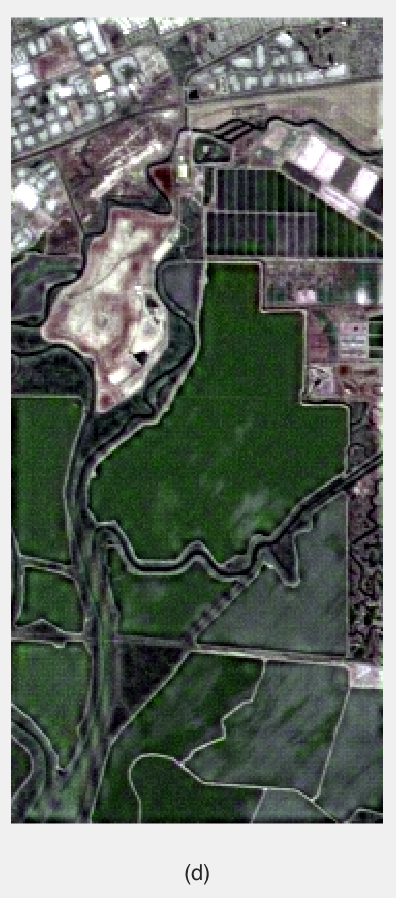}
Figure 7 (a) :  composite color image of the target image $X$; 
 (b) : the PAN image data $Y_M$;\\
    (c) :  composite color image of the HS image data matrix $Y_H$; \\
    (d) :  composite color image of image fusion obtained by using our CG plus Lanzcos and More-Sorenson algorithm.
\end{center}
\end{figure*}

Our algorithm and  the algorithm of Wei, Dobigeon and Tourneret in \cite{WDT} for image fusion differ very little in their data preparation parts. Their only significant variations occur when they solve the associated Sylvester equation which takes up between 50 to 90 \% of  total CPU time for this problem. 
The algorithm of Wei, Dobigeon and Tourneret \cite{WDT} uses interpolation techniques and  a subspace projection method. Its solution matrix $\hat X$ has norm 610.88. It is computed in 1.807 sec overall for this example, with its Sylvester solver using 0.586 sec thereof.\\
 Our CG plus Lanzcos and More-Sorenson algorithm computes the global solution as $ X^*$ with Frobenius  norm 630.89 in 20.9 seconds and its Sylvester solver  with our fast $U\cdot C_2$ implementation takes  19.1 sec of CPU time. Its relative matrix equation error $\|C_1 X^* + X^* C_2 - C_3\|_F/\|X^*\|_F$ is $9.2 \cdot 10^{-10}$ when we set $\Delta = 640$ and $err = 10^{-8}$ inside our Sylvester solver.\\[1mm]
 Clearly our method is much more accurate but more time consuming.  Tightening $err$ increases its accuracy without changing the computed solution $X^*$ except in a few of its trailing digits. To understand the inherent inaccuracy of  Wei, Dobigeon and Tourneret's projection method \cite{WDT} we have used our algorithm with $\Delta$ set equal to 
 610.90 and $err = 10^{-2}$. Then the optimal solution $\tilde X$ inside the norm bounded ball $\{ X \mid ||X\| \leq 610.90\}$  is computed on its boundary  and it has the Frobenius norm 618.94. This low accuracy run took only 5.146 sec and used 2.922 sec for its Sylvester solver part. It achieves the relative matrix equation error of $9.209 \cdot 10^{-4}$ which is a lower bound for the relative error of $\hat X$ obtained in \cite{WDT} . Thus the entries of the interpolation and projection method of solution $\hat X$   from  \cite{WDT}  carry only around 3 accurate digits.\\[1mm]
  This may be good enough for image fusion problems but it gives food for thought otherwise.\\

To further evaluate our method further for image fusion, we compare our approach to seven methods, namely to GSA\cite{ABS2007}, SFIM-HS\cite{L2000}, GLP-HS\cite{AABGS2006}, CNMF\cite{YYI2012}, HySure\cite{SBAC15}, MAP-SMM\cite{E2004} and Wei's FUSE\cite{WDT}. To compare we use the image data of Headwall's Hyperspec Visible and Near-Infrared, series C (VNIR-C) imaging sensor over Chikusei, Ibaraki, Japan, taken in 2014 \cite{YI2016}. Specifically we have selected a $540 \times 420$-pixel-size image set with 128 bands for the experiment, a 2-band MS image set and an HS image set that were obtained respectively by filtering this reference image set and by down-sampling every 5 pixels in both vertical and horizontal directions for each band of the reference image set. We present the experimental  results in table 3 with respect to four quality measures \cite{YGC2017}: \\
1) peak SNR (PSNR) defined as $PSNR(x_i, \hat{x}_i) = 10 \cdot \log_{10} \left( max(x_i)^2 \over \| x_i - \hat{x}_i\|^2_2 / P \right)$, \\
2) spectral angle mapper (SAM) defined as $SAM(x_i, \hat{x}_i) = arccos\left( x^T_j \hat{x}_j \over \| x_j \|_2  \| \hat{x}_j \|_2 \right)$, \\
3) erreur relative globale adimensionnelle de synth\`ese (ERGAS) defined as $ERGAS(X_i, \hat{X}_i) = 100d \sqrt{ {1\over B} \sum^B_{i=1} { \| x_i - \hat{x}_i\|^2_2 \over ({1\over P} 1^T_P X_i)^2}}$, \\
4) $Q2^n$ defined as $Q(x,y) = {4\sigma_{xy} \overline{xy} \over (\sigma^2_x + \sigma^2_y) (\overline{x}^2 + \overline{y}^2)}$.

The quality of the constructed image data is listed in terms of their PSNR, SAM, ERGAS and $Q2^n$. The best results are in bold. These experiments show clearly our CG method, in \cite[fastmult\_SGLTR\_3i\_ABE.m]{UXSE2016} version obtains very satisfactory results.

\begin{table}[t]
   \begin{center} 
\begin{tabular}{|l|c|c|c|c|}
\hline
  Method    & PSNR & SAM &ERGAS & $Q2^n$ \\ \hline
   GSA  & 30.2865 & 3.4605 & \textbf{3.8628}  & 0.82686 \\ \hline  
  SFIM-HS & 22.3387 & 4.1282 & 30.6702 & 0.73535 \\ \hline  
  GLP-HS & 27.2607 & 3.9683 & 4.9433 & 0.76881  \\ \hline
  CNMF    & 26.2803 & 4.6969 & 6.7028 & 0.60606 \\  \hline
  HySure & 26.0945 & 5.6266 & 5.8801 & 0.63306 \\ \hline 
  MAP-SMM & 26.2409 & 4.3180 & 5.7017 & 0.71366 \\ \hline 
  FUSE & 30.9966 & 4.1641 & 4.3650 & 0.78182 \\ \hline 
  Our CG method & \textbf{31.8324} & \textbf{3.2072} & 3.8777 & \textbf{0.8404} \\ \hline 
  \end{tabular}
  \end{center}
 \hspace*{1mm}Table 3: The quality measures for the Hyperspec Chikusel image data.
\end{table}

\subsection{T-Sylvester Matrix Equations.}
An extensive analysis of theoretical and computational aspects of various T-Sylvester type matrix equations was presented by Fr\'oilan Dopico in \cite{FD2013}. Iterative Krylov subspace projection methods and codes for  T-Sylvester equations $A X + X^T B = E$  are available in Dopico et al \cite{DGKS16} for specific low rank right hand side matrices $E$ for which  explicit multi-dyadic representations $ E _{n,n} = C1_{n,r} \cdot (C2_{n,r})^T$ with $r <\!\!< n$ are known a priori. These codes are fast, but for comparisons with our method, huge dense low  rank right hand side matrices $E$ can unfortunately not be handled by our more general iterative method which needs sparse or at least structured matrices in its active $X$ multiplications by  $C_1$ or $C_2$.  \\[-4mm]

\subsection{Outlook.}

Our MATLAB codes are collected at \cite{UXSE2016}  for solving ten different linear matrix equations.
These codes allow for two additional optional inputs apart from the necessary  input matrices: the last optional input $err$ is the desired output accuracy. For image restoration problems for example with relatively low accuracy sensor data when compared to MATLAB's machine constant, an output with a relative error of $10^{-7}$ or $10^{-10}$ for the solution $X$ might suffice rather than our default error bound of $10^{-14}$. Since our algorithm is iterative, to obtain lesser accuracy a lower error threshold will reduce the number of iterations that are needed and this can result in a near 2-fold speed gain. The last but one optional input denotes the maximal norm $\Delta$ of a solution $X$ that we consider for the given Sylvester type equation. Our default is $\Delta = 200$. But for image processing problems, choosing $\Delta=\|E\|$ may be sufficient. If  the computed solution $X$ is such that $\|X\| = \Delta$, then the algorithm has found an optimal norm bounded solution on the boundary of the admissible set $\{ X \mid \|X\| \le \Delta\}$ and there is a chance that 'better' solutions might lie beyond this $\Delta$ sphere. Increasing the $\Delta$ value to two or three times $\Delta$ and repeating the computations may find a solution with smaller relative residual matrix equation error for the given problem. Besides, there is no great efficiency penalty (involving at most just a few extra iterations) if $\Delta$ is chosen not too far above the actual norm of the optimal solution. Finally, our algorithms work equally well for all solvable and unsolvable Sylvester and T-Sylvester type equations and they either find a norm bounded solution if the equation is solvable or they find the optimal norm bounded solution if unsolvable.

As explained and shown earlier, our iterative  methods can easily be adapted and extended to solve any linear matrix equation $f(X) = E$ quickly, accurately and optimally with respect to norm limits for the solution $X$. This can be done for  dense, sparse and certain huge structured matrix systems. It has proven its value as an accuracy checker in a Sylvester matrix equation  example with a massive data matrix \cite{WDT} that previously did not allow for direct accuracy checking of the computed solution.



\section{Conclusion}
In this paper, we document a matrix iterative algorithm that uses the Conjugate Gradient and Lanzcos methods in conjunction with the GLTR algorithm and More-Sorensen's constraint optimization method to solve  inhomogeneous linear matrix equations optimally (Problem (\ref{1.1})). We prove global constrained convergence along two branches in finitely many steps. Throughout we model the Conjugate Gradient Method  for solving Sylvester type equations $f(X) = E \neq 0$ most efficiently through the equations' respective defining functions $f$ and their adjoints $f^*$. All our codes for solving $f(X) = E$ similarly rely on the use of $f$ and $f^*$ and can be easily modified to solve other linear matrix equations.\\
Our method is general and generically applicable to all  linear matrix equations $f(X) = E \neq 0$. It cannot and does not compete with specific applications methods that rely for example on knowledge of low rank factorizations of the right hand side matrix $E$ as \cite{BB13}, \cite{DGKS16}, or \cite{KPT14} do. 
  However, our numerical tests and  real world applications to image fusion and image restoration problems such as encountered in \cite{WBD},  \cite{WDT} ,and \cite{WDTD} illustrate the efficiency, accuracy, and usefulness of our algorithms for solving Sylvester type linear matrix equations.\\[-8mm]

\ifCLASSOPTIONcaptionsoff
  \newpage
\fi



%

 \begin{appendices}

      \section{ Theoretical properties of Algorithm 3.1}
We develop useful  theoretical properties of the matrix sequences of Algorithm 3.1 for the generalized Sylvester equation $f(X) = E \neq 0$. These theoretical properties describe  orthogonality relations between certain computed matrix iterates that will allow us in Section IV to prove general convergence of the method and to speed up  the algorithm further.

\begin{lemma}\label{Lemma1} 
If the matrix sequences $\{R_i \}$, $\{P_i \}$ , and $\{f(P_i)\}$ are
generated as stipulated by Algorithm 3.1, then we have
\begin{equation*}
\langle {R_i ,R_j } \rangle =0,
\quad
\quad
\langle {f(P_i), f(P_j)} \rangle=0.
\end{equation*}
for all $i\ne j$, $0\le i,j\le k$ and $\langle {P_i ,R_j } \rangle =0$ for all $0\le i<j\le k$.
\end{lemma}
\begin{proof}  

We use induction to prove that the conclusion holds for all $0\le i<j\le k$.\\[2mm]
\emph{ Step 1}. We must show that $\langle {R_i ,R_{i+1} } \rangle =0$, $\langle {P_i ,R_{i+1} } \rangle =0$ and $\langle {f(P_i),f(P_j)} \rangle =0$ for all $i=0,1,2,\dots k$. When $i=0$ then $R_0=-P_0$ and we have\\[-2mm]

\small{
\noindent
$
\langle {R_0 ,R_1 } \rangle \ = \ \langle {R_0 ,R_0 +\alpha _0 f^* (f(P_0) )} \rangle = \ \langle {R_0 ,R_0 } \rangle +\frac{\langle {R_0 ,R_0 }
\rangle}{\langle {f(P_0),f(P_0)} \rangle }\langle {R_0 ,f^*( f(P_0))} \rangle $\\
$\hspace*{3mm} = \ \langle {R_0 ,R_0 } \rangle +\frac{\langle {R_0 ,R_0 }
\rangle }{\langle {f(P_0),f(P_0)} \rangle  }\langle {f(R_0),f(P_0)} \rangle \ \ = \ \ 0$\\
\text{and}\\[1mm]
$\langle {P_0 ,R_1 } \rangle  \ = \ \langle {P_0 ,R_0 +\alpha _0 f^* (f(P_0))} \rangle  = \ \langle {P_0 ,R_0 } \rangle  +\frac{\langle {R_0 ,R_0 }
\rangle }{\langle {f(P_0),f(P_0)} \rangle  }\langle
{P_0 , f^*( f(P_0))} \rangle  $\\
$\hspace*{3mm} = \ \langle {P_0 ,R_0 } \rangle  +\frac{\langle {R_0 ,R_0 }
\rangle  }{\langle {f(P_0),f(P_0)} \rangle } \langle {f(P_0),f(P_0)} \rangle  \ \ = \ \ 0 $\\
and finally\\[1mm]
$\langle {f(R_0),f(P_1)} \rangle  \ = \ \langle {f(R_0),f(-R_1+\beta _0 P_0 ) } \rangle  = \  -\langle {f(P_0),f(R_1)} \rangle  +\frac{\langle {R_1 ,R_1
} \rangle }{\langle {R_0 ,R_0 } \rangle  }\langle {f(P_0),f(P_0)} \rangle  $\\
$\hspace*{3mm} = \  -\langle {f^* f(P_0),R_1 } \rangle   \ +  \frac{\langle {R_1,R_1 } \rangle  }{\langle {R_0 ,R_0 } \rangle  }\langle {f(P_0),f(P_0)} \rangle   = \  -\frac{\langle {f(P_0),f(P_0)} \rangle }{\langle {R_0 ,R_0} \rangle  }\langle {R_1 -R_0 ,R_1 } \rangle  + \frac{\langle {R_1 ,R_1 } \rangle  }{\langle {R_0 ,R_0 }
\rangle  }\langle {f(P_0),f(P_0)} \rangle = 0.$
}

\noindent
Assume that the conclusion holds for all $i\le s\ (0<s<k)$. Then\\[1mm]
\small{
$\langle {R_s ,R_{s+1} } \rangle  \ = \ \langle {R_s ,R_s +\alpha _s f^*( f(P_s))} \rangle   = \  \langle {R_s ,R_s } \rangle  +\alpha _s \langle {R_s ,f^* (f(P_s))} \rangle  $\\
$\hspace*{3mm}  = \  \langle {R_s ,R_s } \rangle  +\alpha _s \langle
{-P_s +\beta _{s-1} P_{s-1} ,f^*( f(P_s))} \rangle    = \  \langle {R_s ,R_s } \rangle  +\alpha _s\langle
{-P_s ,f^*(f(P_s))} \rangle \ + \ \alpha _s \langle {\beta _{s-1} P_{s-1} ,f^*(f(P_s))} \rangle  $\\
$\hspace*{3mm}  = \ \langle {R_s ,R_s } \rangle  +\frac{\langle {R_s ,R_s } \rangle  }{\langle {f(P_s),f(P_s)} \rangle }\langle {f(-P_s),f(P_s)} \rangle  \
 = \ 0, $\\[1mm]
and\\[1mm]
$ \langle {P_s ,R_{s+1} } \rangle  \  = \ \langle {P_s ,R_s +\alpha
_s f^*(f(P_s))} \rangle  = \  \langle {P_s ,R_s } \rangle  \ + \frac{\langle {R_s ,R_s }\rangle  }{\langle {f(P_s),f(P_s)} \rangle }\langle
{P_s ,f^*(f(P_s) )} \rangle  $\\
$\hspace*{3mm}  = \  \langle {-R_s +\beta _{s-1} P_{s-1} ,R_s } \rangle 
+\frac{\langle {R_s ,R_s } \rangle  }{\langle {f(P_s),f(P_s)}
\rangle}\langle {f(P_s),f(P_s)}
\rangle = \ 0.$ \\
and\\[1mm]
$\langle {f(P_s),f(P_{s+1})} \rangle  \ =  \ \langle {f(P_s),f(-R_{s+1} +\beta _s P_s )} \rangle   = \ \langle {f(P_s),f(R_{s+1})} \rangle  +\frac{\langle
{R_{s+1} ,R_{s+1} } \rangle  }{\langle {R_s ,R_s } \rangle 
}\langle {f(P_s),f(P_s)} \rangle  $\\
$\hspace*{3mm}  =   -\langle {f^*f(P_s),R_{s+1} } \rangle  \ + \ \frac{\langle {R_{s+1} ,R_{s+1} } \rangle  }{\langle {R_s ,R_s } \rangle }\langle {f(P_s),f(P_s)} \rangle    $\\
$\hspace*{3mm}  =   -\frac{\langle {f(P_s),f(P_s)} \rangle  }{\langle {R_s ,R_s
} \rangle  }\langle {R_{s+1} -R_s ,R_{s+1} } \rangle     +  \frac{\langle {R_{s+1} ,R_{s+1} } \rangle  }{\langle {R_s
, R_s } \rangle  }\langle {f(P_s),f(P_s)} \rangle  =0.$\\[0mm]

}

\noindent
By induction $\langle {P_i ,R_{i+1} } \rangle  =0$, $\langle {R_i ,R_{i+1} } \rangle  =0$ and $\langle {f(P_i),f(P_{i+1})} \rangle =0$
 for all $i=0,1,2,\dots, k$.\\[2mm]
\noindent
\emph{Step 2}. We use that $\langle {P_i ,R_{i+l} } \rangle  =0$,
$\langle {f(P_i),f(P_{i+l})} \rangle  =0$ and $\langle {R_i ,R_{i+l} } \rangle  =0$ for all $0\le i\le k$ and $1<l<k$ and show that $\langle {P_i ,R_{i+l+1} } \rangle  =0$, $\langle {f(P_i),f(P_{i+l+1})} \rangle $ $=0$ and $\langle {R_i ,R_{i+l+1} }\rangle  =0$:\\[2mm] 
$\langle {P_i ,R_{i+l+1} } \rangle  \ = \ \langle {P_i ,R_{i+l}
+\alpha _{i+l} f^*(f(P_{i+l}))} \rangle    = \ \langle {P_i ,R_{i+l} } \rangle  +\alpha _{i+l}\langle
P_i ,f^*(f(P_{i+l})) \rangle   = \ \alpha _{i+l} \langle f(P_i),f(P_{i+l}) \rangle  \ = \ 0. $\\
$\langle {f(P_i),f(P_{i+l+1})} \rangle  \ =\langle {P_i,f^*(f(P_{i+l+1}))} \rangle  = \frac{1}{\alpha _{i+l} }\langle {P_i ,R_{i+l+1} -R_{i+l} }
\rangle  \ = \ 0. $\\[1mm] 
\noindent
and\\[1mm]     
$\langle {R_i ,R_{i+l+1} } \rangle  \ = \langle {R_i ,R_{i+l}
+\alpha _{i+l} f^*(f(P_{i+l}))} \rangle  = \alpha _{i+l} \langle {R_i ,f^*(f(P_{i+l}))} \rangle $ \\
$\hspace*{3mm} = \alpha _{i+l} \langle {-P_i +\beta _{i-1} P_{i-1} , f^*(f(P_{i+l}))}\rangle  = \alpha _{i+l} (\langle {-f(P_i),f(P_{i+l})} \rangle  +\beta
_{i-1} \langle {f(P_{i-1}),f(P_{i+l})} \rangle ) = 0$ \\[1mm]  
From Steps 1 and Step 2, we learn by  induction that 
\[
\langle R_i , R_j \rangle  =0, \quad \langle {P_i ,R_j } \rangle  =0,\quad \langle {f(P_i),f(P_j)} \rangle  =0
\]
for $0\le i<j\le k$. Since $\langle {A,B} \rangle =\langle {B,A} \rangle $  holds for all same size matrix pairs $A$ and $B$, $\langle R_i , R_j \rangle  =0, \langle {f(P_i),f(P_j)} \rangle  =0$ also hold for all $0\le j<i\le k$ and our proof is complete.

\end{proof}

\begin{lemma}\label{Lemma5}

The sequence of matrices $X_k$ generated by Algorithm 3.1 satisfies 
\begin{equation*}
0=\|X_0\|< \cdots <\|X_k\|<\|X_{k+1}\| \ .
\end{equation*}

\end{lemma}
\begin{proof}

We first show that the matrix sequences generated by Algorithm 3.1 satisfy $\langle {X_k, R_k \rangle  }=0$ for $k \geqslant 0$ and $\langle {X_k, P_k \rangle  }=0$ for $k\geqslant1$.\\[2mm]
Our algorithm computes $X_{k+1}$ recursively in terms of $X_{k}$. Once all the terms of this recursion are written out explicitly, we have
\begin{equation*}
X_k \ =\ X_0+\sum^{k-1}_{i=0}\alpha_i P_i \ = \ \sum^{k-1}_{i=0}\alpha_i P_i
\end{equation*}
since $X_0=0$. Taking the inner product with $R_j$ and applying Lemma \ref{Lemma1} gives us
\begin{equation*}
\langle {X_k, R_k \rangle  } \  =  \ \langle { \sum^{k-1}_{i=0}\alpha_i P_i, R_k \rangle  } \ = \ \sum^{k-1}_{i=0}\alpha_i\langle {  P_i, R_k \rangle  } \ = \ 0.
\end{equation*}

An induction proof  will  establish the second assertion that $\langle{  X_k, P_k \rangle  }>0$. To do so we apply  Lemma \ref{Lemma1} again and obtain
\begin{equation}\label{16}
\langle{  X_1, P_1 \rangle  }  =  \langle{ \alpha_0 P_0, -R_1+\beta_0 P_0 \rangle  }  = \alpha_0 \beta_0 \langle{  P_0,  P_0 \rangle  }  > 0.
\end{equation}
We now make the inductive hypothesis that  $\langle{  X_k, P_k \rangle  }>0$ and show that this implies $\langle{  X_{k+1}, P_{k+1} \rangle  }>0$. From (\ref{16}), we have $\langle{  X_{k+1}, R_{k+1} \rangle  }=0$, and therefore we have
\small{\begin{eqnarray*}
\langle{  X_{k+1}, P_{k+1} \rangle  }=\langle{  X_{k+1}, -R_{k+1}+\beta_{k}P_{k} \rangle  } 
   =   \beta_{k} \langle{  X_{k+1}, P_k \rangle  }= \\
  =  \beta_{k} \langle{  X_{k}+\alpha_k P_k, P_k \rangle  } =\beta_{k} \langle{  X_{k}, P_k \rangle  }+\alpha_k \beta_{k} \langle{ P_k, P_k \rangle  }.
\end{eqnarray*}}
And by the inductive hypothesis the last expression is positive.\\[-3mm]

Next we prove that $\|X_k\|<\|X_{k+1}\|$, where $X_{k+1} = X_k + \alpha_k P_k$ and $k \geqslant 1$. Observe that \\[3mm]
$$
\|X_{k+1}\|^2=\langle{ X_k+\alpha_k P_k, X_k + \alpha P_k \rangle } = \|X_{k}\|^2 + 2\alpha_k \langle{ X_k, P_k \rangle } +\alpha^2_k \|P_{k}\|^2 > \|X_{k}\|^2 + \alpha^2_k \|P_{k}\|^2$$
This shows  that $\|X_k\|< \|X_{k+1}\|$ and our proof is complete.
\end{proof}

\begin{remark}\label{Remark2}
In Lemma \ref{Lemma1}, the matrix sequence $R_0 ,R_1 ,R_2 ,\dots \subset \R^{m\times n}$ is mutually orthogonal. Therefore there is a positive number $k+1\le m\cdot n $ with $R_{k+1} =0$. Hence, disregarding rounding errors, as long as the algorithm never switches to the second branch, the first stopping criterion of our algorithm will be satisfied after finitely many iterations.
\end{remark}

\begin{lemma}\label{Lemma2} \cite[Lemma 2]{XXP15}
Let $\{Q_i \}$ be the matrix sequence generated by Algorithm 3.1.
    Then this matrix sequence consists of mutually  orthonormal matrices in the Frobenius norm.
\end{lemma}
\noindent
The proof below replicates the one given for the matrix equation $AXB = C$ {with symmetric constraint} in \cite{XXP15}. We rework it here for clarity, now with the generalized Sylvester equation with its any number of terms function $f(X)$ in \ref{1.0} instead of just $AXB$ in \cite{XXP15}.\\[1mm]
Recall that the matrix inner product of two matrices in $\R^{m\times n}$ is defined as $\langle A,B \rangle =\tr  (B^T A)$.

\begin{proof} By  definition  $\langle {Q_i ,Q_i } \rangle  =1$ for  all $i = 0,1,2,\dots $.  And we again use induction in the following two steps:  \\[1mm]
\emph{Step 1}. Show that $\langle {Q_i ,Q_{i+1} } \rangle  =0$ for all
$i=0,1,2,\dots k$: When $i=0$, we have\\[2mm]
$\langle {Q_0 ,Q_1 } \rangle  \ =  \  \frac{1}{\gamma _1 }\langle
{Q_0,f^*(f(Q_0))-\delta _0 Q_0 } \rangle   = \  \frac{1}{\gamma _1 }(\langle {Q_0 ,f^*(f(Q_0))} \rangle -\langle {Q_0 ,\delta _0 Q_0 } \rangle  ) $\\
$\hspace*{3mm}  = \  \frac{1}{\gamma _1 }(\langle {f(Q_0),f(Q_0) } \rangle 
-\langle {f(Q_0),f(Q_0)} \rangle  \cdot \langle {Q_0 ,Q_0 } \rangle  )= \ 0.$ \\[1mm]
Assume that the conclusion hold for all $i\le s\ (0<s<k)$. Then\\[2mm]
$\langle {Q_s ,Q_{s+1} } \rangle  \  = \  \frac{1}{\gamma _{s+1}
}\langle {Q_s ,f^*(f(Q_s)) -\delta _s Q_s -\gamma _s Q_{s-1} } \rangle    = \ \frac{1}{\gamma _{s+1} }( \langle Q_s ,f^*(f(Q_s)) \rangle  -\langle {Q_s ,\delta _s Q_s } \rangle  \ -\langle {Q_s ,\gamma_s Q_{s-1} } \rangle  ) $\\
$\hspace*{3mm}  = \ \frac{1}{\gamma _{s+1} }(\langle {f(Q_s),f(Q_s B)} \rangle  -\langle {f(Q_s),f(Q_s)} \rangle  \cdot \langle {Q_s ,Q_s } \rangle   -\gamma _s \langle {Q_s ,Q_{s-1} } \rangle  )=0$\\[1mm]
Hence by induction, $\langle {Q_i ,Q_{i+1} } \rangle  =0$ holds for all $i=0,1,2,\dots k$.\\[2mm]
\emph{Step 2}. Assume that $\langle {Q_i ,Q_{i+l} } \rangle  =0$ for all
$0\le i\le k$ and $1<l<k$. Then\\[2mm]
$\langle {Q_i ,Q_{i+l+1} } \rangle  \  = \ \frac{1}{\gamma _{i+l+1}}\langle Q_i ,f^*(f(Q_{i+l})) -\delta _{i+l} Q_{i+l} -\gamma _{i+l}Q_{i+l-1}  \rangle  $\\
$\hspace*{3mm}  = \ \frac{1}{\gamma _{i+l+1} }(\langle {f(Q_i),f(Q_{i+l})} \rangle  -\langle {Q_i ,\delta _{i+l} Q_{i+l} } \rangle  -\langle {Q_i,\gamma _{i+l} Q_{i+l-1} }\rangle  ) $\\
$\hspace*{3mm}  = \ \frac{1}{\gamma _{i+l+1} }\langle {f(Q_i),f(Q_{i+l})}\rangle   = \ \frac{1}{\gamma _{i+l+1} }\langle {f^*f(Q_{i})} \rangle  $\\
$\hspace*{3mm}  = \ \frac{1}{\gamma _{i+l+1} }\langle {-\gamma _{i+1} Q_{i+1} -\delta _i Q_i -\gamma _i Q_{i-1} ,Q_{i+l} } \rangle  \ = \ 0.$\\[1mm]
Steps 1 and 2 prove that $\langle {Q_i ,Q_j } \rangle =0$ for all $i,j=0,1,2,\dots k,i\ne j$.  \end{proof}


\begin{lemma}\label{Lemma3} \cite[Lemma 3]{XXP15}
Let  $\{\gamma _k \}$, $\{T_k \}$ and $\{Q_i
\}$ be the sequences generated by Algorithm 3.1. Let
\small{
\begin{eqnarray*}
\tilde {X}=Q_0 h^0+Q_1 h^1+...+Q_k h^k =(Q_0 ,Q_1 ,\dots ,Q_k )(h\otimes
I) ,
h=(h^0,h^1,\dots ,h^k)^T\in \R^{k+1},
\end{eqnarray*}}
Then
\begin{equation*}
\frac{1}{2}\langle {f(\tilde {X}),f(\tilde {X})} \rangle 
-\langle {f(\tilde {X}),E} \rangle  =\frac{1}{2}h^TT_k h+\gamma _0
h^Te_1 ,
\end{equation*}
where $e_1 $ is the first unit vector and $T_k$ is positive semi-definite.
\end{lemma}

\begin{proof} By the definition of $T_k $ and $Q_k$  $(k=0,1,2,\dots )$, we have
\begin{equation*}
\begin{array}{l}
(M_0,M_1,\dots,M_k)
=(Q_0 ,Q_1 ,\dots ,Q_k)(T_k \otimes I)+(0,\dots ,0,\gamma _{k+1} Q_{k+1} ),
\end{array}
\end{equation*}
where $M_i=f^*(f(Q_i))$, $0\leqslant i\leqslant k$.
Hence, we have\\[2mm]
\small{
\hspace*{-2mm}
$\langle {f(\tilde {X}),f(\tilde {X})} \rangle  \ =\langle{\tilde {X},f^*(f(\tilde {X}))} \rangle  =$ \\[2mm]
\hspace*{-3.2mm}
$\hspace*{3mm} =\langle {(Q_0 ,Q_1 ,\dots, Q_k )(h\otimes I),f^*f(Q_0 h^0+Q_1 h^1+\dots+Q_k h^k)} \rangle    $\\[2mm]
\hspace*{-3mm}
$\hspace*{3mm} =\langle {(Q_0 ,Q_1 ,\dots, Q_k )(h\otimes I),M_0h^0+M_1h^1+\dots+M_kh^k} \rangle $ \\[2mm]
\hspace*{-3mm}
$\hspace*{3mm} =\langle {(Q_0 ,Q_1 ,\dots, Q_k )(h\otimes I),(M_0,M_1,\dots,M_k)(h\otimes I)} \rangle $\\[2mm]
\hspace*{-3mm}
 $\hspace*{3mm} =\langle {(Q_0 ,Q_1 ,\dots, Q_k )(h\otimes I),(Q_0 ,Q_1,\dots, Q_k )(T_k \otimes I)(h\otimes I)} \rangle $\\
 \hspace*{-3mm}
$\hspace*{3mm} = \tr[(Q_0 ,Q_1 ,\dots, Q_k )(h^TT_k h\otimes I)(Q_0 ,Q_1 ,\dots,
Q_k )^T] $\\
\hspace*{-3mm}
 $\hspace*{3mm} =h^TT_k h.$ \\[1mm]
}
and \\[1mm]
$\langle {f(\tilde {X}),E} \rangle=\langle {\tilde {X},f^*(E)} \rangle =-\langle {Q_0 h^0+Q_1 h^1+...+Q_k h^k,\gamma _0 Q_0 } \rangle=-\gamma _0 \langle h^0Q_0,Q_0 \rangle=-\gamma _0h^0=-\gamma _0 h^Te_1$\\

\noindent
Therefore the equation holds. And obviously $h^TT_k h =\langle {f(\tilde {X}),f(\tilde {X})} \rangle \ge 0$ for all $h \in \R^{k+1}$. Therefore  $T_k$ is positive semi-definite and the proof is complete.  \end{proof}


\begin{theorem}\label{Theorem2} \cite[Theorem 2]{XXP15}
 Assume that the sequences $\{Q_k \}$, $\{R_k \}$, $\{\gamma _k \}$, $\{\delta _k \}$, $\{\alpha _k \}$ and $\{\beta _k \}$ are generated by Algorithm 3.1. Then the following equations hold for all $k=0,1,2,\dots $.
 
 \small{
\begin{equation}\label{7}
\hspace*{-3mm}
Q_k =(-1)^k\frac{R_k }{\left\| {R_k } \right\|},
\delta _k =\left\{ {\begin{array}{l} \textstyle{1 \over {\alpha _k }},k=0, \\
 \textstyle{1 \over {\alpha _k }}+\textstyle{{\beta _{k-1} } \over {\alpha
_{k-1} }},k>0, \\
 \end{array}} \right.
\gamma _k =\frac{\sqrt {\beta _{k-1} } }{\alpha _{k-1} }
\end{equation}}
\end{theorem}
\noindent

\begin{proof}  By the definition of $Q_k $
and $R_k $, we have
\small{
\begin{equation}\label{8}
\begin{array}{lll}
 \hspace*{-3mm}Q_k = a_k H^k f^*(E)+a_{k-1}
H^{k-1}f^*(E)  +  \ a_0 f^*(E),
\end{array}
\end{equation}}
\small{
\begin{equation}\label{9}
\begin{array}{lll}
R_k =(-1)^kb_k H^k f^*(E) +  (-1)^{k-1}b_{k-1}
H^{k-1} f^*(E)+ \dots  +b_0 f^*(E),
\end{array}
\end{equation}}
where the $a_i $ and $b_i$  (for $i=0,1,2,\dots ,k)$ are real numbers and $H=f^* \circ f$. These equations imply that $Q_k $ and $R_k $ belong to the space 
\begin{eqnarray}\label{9'}
K_k =span\left\{ H^k f^*(E),H^{k-1}f^*(E),\dots, f^*(E) \right\}.
\end{eqnarray}
Furthermore, we have
\begin{eqnarray*}
span\left\{ {Q_{k-1} ,Q_{k-2} ,\dots ,Q_0 } \right\}=K_{k-1} 
=span\left\{ {R_{k-1} ,R_{k-2} ,\dots ,R_0 } \right\}.
\end{eqnarray*}
By Lemmas \ref{Lemma1} and \ref{Lemma2} we have 
\begin{equation}\label{17}
Q_k \bot K_{k-1} \ \  \text{and} \ \  R_k \bot K_{k-1}.
\end{equation}
Hence, $Q_k $ and $R_k $ must be linear correlation, or $Q_k =c_k R_k$ since $Q_k + K_{k-1}=R_k + K_{k-1}=K_k$.  
So, there exists a real number $c_k $ such that $Q_k =c_k R_k $. Noting that $\left\| {Q_k }
\right\|=1$, we have by (\ref{8}) and (\ref{9}) that $$Q_k =(-1)^k R_k /\left\| {R_k } \right\|,$$ viz $Q_k = (-1)^k \frac{R_k}{\left\| R_k \right\|}$. This establishes the first equation of \eqref{7}.
Noting that the first equation in (\ref{7}) holds we have for $k = 0$
\begin{eqnarray*}
\delta _0 &=&\langle {f(Q_0),f(Q_0)} \rangle 
=\langle {f(R_0),f(R_0)} \rangle   /\|R_0\|^2 
=\langle {f(P_0),f(P_0) } \rangle   /\|R_0\|^2 \ = \ 1/\alpha _0 .
\end{eqnarray*}
And for  $k>0$ we have\\[-3mm]

\small{
$\hspace*{-3mm}\delta _k \ = \ \langle {f(Q_k),f(Q_k)} \rangle  = \ \langle {f(R_k),f(R_k)} \rangle  /\|R_k\|^2 $\\
$ \hspace*{2mm} = \ \langle {f(-P_k +\beta _{k-1} P_{k-1} ),f(-P_k +\beta _{k-1} P_{k-1})} \rangle  / \| R_k \|^2   + \beta _{k-1}^2\langle {f(P_{k-1}),f(P_{k-1})} \rangle  /\|R_k\|^2 $\\
$ \hspace*{2mm} = \ 1/\alpha _k +\beta _{k-1} ^2\langle {f(P_{k-1}),f(P_{k-1})}
\rangle   /\|R_k\|^2 = \ 1/\alpha _k +\beta _{k-1}\langle {f(P_{k-1}),f(P_{k-1})}
\rangle   /\|R_{k-1}\|^2 $\\
$\hspace*{2mm} = \ 1/\alpha _k +\beta _{k-1} /\alpha _{k-1}. $\\[-2mm]
}

\noindent \normalsize
We have just established the second equation in \eqref{7}.
By the definition of $\gamma _k $, we have\\[-2mm]

\small{
$\hspace*{-3mm}\gamma _k ^2 \ = \ \langle {t_k ,t_k } \rangle  = \ \langle f^*f(Q_{k-1}) -\delta _{k-1} Q_{k-1} - \gamma _{k-1} Q_{k-2} ,\gamma_k Q_k \rangle = \ \gamma _k \langle {f^*f(Q_{k-1}),\gamma_k Q_k } \rangle  $\\
$ \hspace*{2mm} = \ - \ \gamma _k \langle {f^*f(R_{k-1}),R_k} \rangle  /(\left\|{R_{k-1} } \right\| \left\| {R_k } \right\|)= \ -\gamma _k \langle{ f^*f(-P_{k-1} +\beta _{k-2} P_{k-2} ),R_k }
\rangle  /(\left\| {R_{k-1} } \right\|\left\| {R_k } \right\|)$ \\
$ \hspace*{2mm} = \ {\gamma _k \langle {(R_{k-1} -R_k )/\alpha _{k-1} +(\beta _{k-2}
/\alpha _{k-2} )(R_{k-1} -R_{k-2} ),R_k } \rangle \over (\left\| {R_{k-1} }
\right\| \left\| {R_k } \right\|) }  = \ \gamma _k \frac{\sqrt {\beta _{k-1} } }{\alpha _{k-1} }.$ \\[-1mm]
}

\noindent \normalsize
Hence the third equation in \eqref{7} holds and the proof is complete.  
\end{proof}

\begin{remark}\label{Remark1}
Theorem \ref{Theorem2} relates  the sequences $\{Q_k \}$,
$\{R_k \}$, $\{\gamma _k \}$, $\{\delta _k \}$, $\{\alpha _k \}$ and $\{\beta
_k \}$. This will be used to reduce the cost of our calculation in Section III.
\end{remark}

\begin{lemma}\label{Lemma6}

In the first branch of Algorithm 3.1, $X_{k+1}$ is a solution of the problem
 \begin{equation}\label{15}
\mathop {\min}\limits \psi(X) = \frac{1}{2}\langle  {f(X),f(X)} \rangle  -\langle {f(X),E}\rangle ,
\end{equation}
 where $X=(Q_0 ,Q ,\dots ,Q_k )(h \otimes I)$ with $h \in \R^{k+1}$, or equivalently for all $X \in K_k$ as defined in (\ref{9'}).
\end{lemma}
\begin{proof}
By Lemmas \ref{1.1} and \ref{1.2} we have 
\begin{equation}\label{17a}
Q_k \bot K_{k-1} \ \  \text{and} \ \  R_k \bot K_{k-1}.
\end{equation}
Using this equation we have\\[-2mm]

\small{
$\psi(X_{k+1}+W) \ =$\\[2mm]
$\hspace*{6mm} \ = \ \frac{1}{2}\langle  {f(X_{k+1}+W),f(X_{k+1}+W)} \rangle   - \langle {f(X_{k+1}+W),E}\rangle $\\[3mm]
$\hspace*{6mm} \ = \ \left\{ \frac{1}{2} \langle  {f(X_{k+1}),f(X_{k+1})} \rangle  - \langle  {f(X_{k+1}),E} \rangle  \right\}   + \ \langle  {f(W),f(X_{k+1})} \rangle  - \langle  {f(W),E} \rangle   + \  \frac{1}{2} \langle  {f(W),f(W)} \rangle $\\[3mm]
$\hspace*{6mm} \ = \ \psi(X_{k+1}) +  \langle  {W, f^*(f(X_{k+1}))} \rangle    - \langle  {W, f^*(E)} \rangle  + \frac{1}{2} \langle  {f(W),f(W)} \rangle $\\[2mm]
$\hspace*{6mm} \ = \  \psi(X_{k+1}) +  \langle  {W, R_{k+1}} \rangle + \frac{1}{2} \langle  {f(W),f(W)} \rangle $\\[3mm]
$\hspace*{6mm} \ = \   \psi(X_{k+1}) + 0 + \frac{1}{2} \langle  {f(W),f(W)} \rangle  \ \geqslant \ \psi(X_{k+1})$\\[2mm]
}
for all $W \in K_k$. Therefore our claim is established.
\end{proof}

\normalsize
\begin{theorem}\label{Theorem5}
At least one of the solutions $h_k$ of Problem (\ref{6}) from the second branch of Algorithm 3.1
lies on the boundary of $\left\{h : \|h\|_2 \le \Delta\right\}$. This $h_k$ solves  the 
optimization problem   
\begin{equation}\label{12}
\hspace*{-3mm} \mathop {\min }\limits_{h\in \R^{k+1}} \frac{1}{2}h^TT_k h+h^T(\gamma _0 e_1)
 \mbox{\rm subject to} \| h \|_2 =\Delta. 
\end{equation}
\end{theorem}

\begin{proof} Assume that every solution $h_k$ of Problem (\ref{6}) lies in the open set $\left\{h : \|h\|_2 < \Delta\right\}$.  Then, according to \cite[Theorem 2]{XXP15} as rephrased in Theorem \ref{Theorem1} in the next section, there exists a nonnegative number $\lambda _k $ such that 
\begin{eqnarray}\label{2.8}
(T_k +\lambda _k I)h_k =-\gamma _0 e_1 ,\nonumber\\
 \text{ with } \lambda _k \cdot ( \left\| {h_k } \right\|_2-\Delta )=0 \  \text{and}  \  \left\| {h_k } \right\|_2 \le \Delta.
\end{eqnarray} 
Since $\|h_k\|_2 < \Delta$  the second formula of (\ref{2.8}) implies that $\lambda _k = 0$. And  the first formula of (\ref{2.8}) then ensures  $T_k h_k =-\gamma_0 e_1 $. \\[1mm]
By Lemma \ref{Lemma3}, $T_k $ is  positive semidefinite.\\[2mm]
Case (1): We  prove that if $T_k $ is positive definite, then solving problem (\ref{6}) in the second branch of Algorithm 3.1 cannot occur, leading to a contradiction. To show that the first branch (CG method) has lead to success in this case we show that $P_j \ne 0$ and $f(P_j) \ne 0$ for all $j=0,1,2,\cdots,k$ and $\|X_j\|<\Delta$ for all $j=1,2,\cdots,k+1$.\\[1mm]
Since  $T_k $ is positive definite, its diagonal entries $\delta_j \ne 0$  for all $j=0,1,2,\cdots,k$. Then $f(Q_j) \ne 0$ and thus $Q_j \ne 0$. From Theorem \ref{Theorem2} we know that $R_j \ne 0$. Since $P_{j} =-R_{j} +\beta _{j-1} P_{j-1}$ and $R_j \bot P_{j-1}$, we conclude that $P_j \ne 0$ for all $j=1,2,\cdots,k$.\\[1mm]    

From the proof of Lemma \ref{Lemma3}, we have $$\langle {f(X),f(X)} \rangle  = h^T T_k h > 0$$ for all $ X=(Q_0 ,Q_1 ,\dots ,Q_k )(h \otimes I)\neq 0$, or equivalently for all $ X\in K_k \setminus 0$. Since {$P_k \in K_k \setminus 0$}, we have $\langle {f(P_k),f(P_k)} \rangle  > 0$. Then $f(P_k) \ne 0$. 
Since  for  $j=0,1,2,\cdots,k$ each $T_j $ is positive definite  as a submatrix of $T_k $, we conclude that   $f(P_j) \ne 0$ for all $j$. 
As $T_k $ is positive definite,  $h_k =-T_k^{-1} (\gamma _0 e_1 )\neq 0$ with
$  \| h_k \|_2 <\Delta $. Thus $h_k$  is also a solution of Problem (\ref{6}).
Clearly $h_k =-T_k^{-1} (\gamma _0 e_1 )$ must be unique as the solution of $\mathop {\min }\limits_{h\in \R^{k+1}} \frac{1}{2}h^TT_k h+h^T(\gamma _0 e_1)$. This  combined with Lemma \ref{Lemma3} states  that $\tilde{X}_k=(Q_0 ,Q_1 ,\dots ,Q_k )(h_k \otimes I)$ is the unique solution of Problem (\ref{15}). From Lemma \ref{Lemma6}, we have $ X_{k+1}= \tilde{X}_k$ with $X_{k+1}$  from the CG method part. Since $\| {\tilde{X}_k} \|= \| h_k\|_2<\Delta $, $ \|X_{k+1}\|<\Delta$. From Lemma \ref{Lemma5}, we know that $\|X_1\| <\cdots   < \|X_j\| <\cdots < \|X_{k+1}\|<\Delta$. Therefore in its first $k$ iteration steps, Algorithm 3.1 has only been implemented inside the first, the CG  branch, which is a contradiction.\\[2mm]
Case (2). If $T_k $ is positive
semidefinite but not definite, then there exists a vector $z$ such that $T_k (\tilde{h}_k +z)=-\gamma _0 e_1 $ and  $\| \tilde{h}_k +z \|=\Delta $. This implies that $\tilde{h}_k +z = h_k$ is also a solution of Problem (\ref{6}) on the boundary.   \end{proof}


      \section{Proof of Theorem \ref{Theorem1} and Lemma \ref{Lemma4}}
 \subsection{Proof of Theorem \ref{Theorem1}}
The proof of Theorem \ref{Theorem1} extends the proof of Theorem 2 in \cite{XXP15} that was  given there for explicit 1-term Sylvester functions, to multi-term ones and it is now re-formulated in terms of $f$ and its  adjoint function $f^*$.\\[-4mm]

\begin{proof} Assume that there is a scalar $\lambda ^\ast \ge 0$ such that (\ref{3}) holds. Define 
\begin{eqnarray*}
\varphi (X)&=&\frac{1}{2}\langle {f(X),f(X)} \rangle  -\langle {f(X),E} \rangle , \ \ \text{and} \\
 \mathop \varphi \limits^\wedge (X) &=& 
  \frac{1}{2}\langle {f(X),f(X)} \rangle   +\frac{1}{2}\lambda ^\ast \cdot \langle {X,X} \rangle -\langle {f(X),E} \rangle  
 = \varphi (X)+\frac{1}{2}\lambda ^\ast  \cdot \langle {X,X} \rangle .
\end{eqnarray*}

\noindent
For any matrix $W\in \R^{m\times n}$, we have\\[2mm]
\small{
$\hspace*{15mm} \mathop \varphi \limits^\wedge (X_\ast +W)\ = $\\[-4mm]
\begin{eqnarray*}
&=& \frac{1}{2}\langle
{f(X_\ast +W),f(X_\ast +W)} \rangle  +\frac{1}{2}\lambda ^\ast
\langle {(X_\ast +W),(X_\ast +W)} \rangle  -\langle {f(X_\ast+W),E} \rangle  \\
&=&\left\{ \frac{1}{2}\langle {f(X_\ast),f(X_\ast)} \rangle 
+ \frac{1}{2}\lambda ^\ast \langle {X_\ast ,X_\ast } \rangle  -\langle {f(X_\ast),E} \rangle  \right\}\\
&&\quad +\langle {f(W),f(X_\ast)}\rangle   +\lambda ^\ast \langle {W,X_\ast } \rangle 
-\langle {f(W),E} \rangle  + \frac{1}{2}\langle {f(W),f(W)}\rangle +\frac{1}{2}\lambda ^\ast\langle {W,W} \rangle  \\
&=&\mathop \varphi \limits^\wedge (X_\ast )+\langle {W,(f^*f(X_\ast) -f^*(E) +\lambda ^\ast X_\ast)}\rangle  +\frac{1}{2}\langle {f(W),f(W)} \rangle  +\frac{1}{2}\lambda ^\ast \langle {W,W}\rangle \\
&=&\mathop \varphi \limits^\wedge (X_\ast )+\frac{1}{2}\langle {f(W),f(W)}\rangle 
+\frac{1}{2}\lambda ^\ast \langle {W,W} \rangle \ \ge \
\mathop \varphi \limits^\wedge (X_\ast ). \\[-3mm]
\end{eqnarray*}}
This implies that $X_\ast$ is a global minimizer of the function $\mathop \varphi \limits^\wedge (X)$. Since $\mathop \varphi \limits^\wedge (X)\ge
\mathop \varphi \limits^\wedge (X_\ast )$ for all $X\in \R^{m\times n}$, we
have
\[
\varphi (X)\ge \varphi (X_\ast )+\frac{1}{2}\lambda ^\ast (\langle
{X_\ast ,X_\ast } \rangle  -\langle {X,X} \rangle ).
\]
Now $\lambda ^\ast (\left\| {X_\ast } \right\|-\Delta )=0$ implies
that $\lambda ^\ast (\left\| {X_\ast } \right\|-\Delta )\cdot  (\left\| {X_\ast } \right\|+\Delta ) = \lambda ^\ast (\langle {X_\ast ,X_\ast } \rangle  -\Delta
^2)=0$. Consequently,  
\[
\varphi (X)\ge \varphi (X_\ast )+\frac{1}{2}\lambda ^\ast (\Delta
^2-\langle {X,X} \rangle  )
\]
always holds.
Hence for $\lambda ^\ast \ge 0$ we have $\varphi (X)\ge \varphi (X_\ast )$
for all $X\in \R^{m\times n}$ with $\left\| X \right\| \le \Delta$. Therefore $X_\ast $ is a global minimizer of (\ref{1.2}).

Conversely assume that $X_\ast $ is a global solution of Problem (\ref{1.2}).
We show that there is a nonnegative $\lambda ^\ast $ such that satisfies (\ref{3}).\\
We consider two cases, that $\left\| {X_\ast }
\right\| <\Delta$ or that $\left\| {X_\ast } \right\|=\Delta$.\\
If $\left\| {X_\ast } \right\|<\Delta$, then $X_\ast $ is  an
unconstrained minimizer of $\varphi (X)$ and $X_\ast $ satisfies the
stationary point condition $\nabla \varphi (X_\ast )=0$, that is 
\[
f^*(f(X_\ast)) - f^*(E)=0.
\] This implies that (\ref{3}) holds for $\lambda ^\ast =0$.\\
 When $\left\| {X_\ast } \right\|=\Delta $, the second equation of  (\ref{3}) is
 satisfied and consequently $X_\ast $ is  the solution of the constrained problem
\begin{equation*}
\mathop {\min }\limits_{X\in \R^{m\times n}} \varphi (X)
\quad \mbox{subject to}\quad \left\|X \right\|=\Delta.
\end{equation*}
By applying the optimality conditions for constrained optimization to this problem, we know that there exists a scalar $\lambda ^\ast $ such that the Lagrangian function defined by
\begin{equation*}
\zeta (X,\lambda )=\varphi (X)+\frac{1}{2}\lambda (\langle {X,X} \rangle  -\Delta ^2)
\end{equation*}
has a stationary point at $X_\ast $. By setting $\nabla _X \zeta (X_\ast,\lambda ^\ast )$ equal to zero we obtain
\begin{equation}\label{4}
f^*(f(X_\ast))-f^*(E)+\lambda ^\ast X_\ast =0.
\end{equation}
The proof is finished by showing that $\lambda ^\ast \ge 0$. Since  equation (\ref{4}) holds, $X_\ast $ minimizes $\mathop \varphi \limits^\wedge (X)$. Therefore we have
\begin{equation}\label{5}
\varphi (X)\ge \varphi (X_\ast )+\frac{1}{2}\lambda ^\ast (\langle
{X_\ast ,X_\ast } \rangle  -\langle {X,X} \rangle  )
\end{equation}
for all $X\in \R^{m\times n}$. Suppose that there are only negative values
of $\lambda ^\ast $ that satisfy (\ref{4}). Then we have from (\ref{5}) that
\begin{equation*}
\varphi (X)\ge \varphi (X_\ast )
\quad \mbox{whenever} \left\| X \right\| \ge \left\|
{X_\ast } \right\|=\Delta .
\end{equation*}
Since we already know that $X_\ast $ minimizes $\varphi (X)$ for $\left\| X
\right\| \le \Delta $, it follows that $X_\ast $ is a global, i.e.,  unconstrained minimizer of $\varphi (X)$.
Therefore, condition (\ref{4}) holds with $\lambda ^\ast =0$,
which contradicts our assumption that only negative values of $\lambda ^\ast $
can satisfy  (\ref{4}).   \end{proof}

 \subsection{Proof of Lemma \ref{Lemma4}}

\begin{proof} Assume that $h_k $ is the solution of Problem (\ref{6}). Then
there exists a nonnegative number $\lambda _k $ such that the following
identities hold:
\small{
\begin{equation}\label{11}
(T_k +\lambda _k I)h_k =-\gamma _0 e_1, 
 \lambda _k \cdot (\left\| {h_k } \right\|_2 -\Delta )=0 \  \text{and} \ \left\| {h_k } \right\|_2 \le \Delta.
\end{equation}    }
Since $\| {\tilde {X}_k } \|=\left\| {h_k } \right\|_2$  by Lemma \ref{Lemma2} the second and third identities  in (\ref{10}) hold. The first equation in (\ref{11}) can be rewritten as
\begin{eqnarray*}
(T_k \otimes I)(h_k^0 I,h_k^1 I,\dots, h_k^k I)^T+\lambda _k (h_k^0 I,h_k^1
I,\dots, h_k^k I)^T  +(\gamma _0 I,0,\dots, 0)^T=0,
\end{eqnarray*}
and thus
\begin{eqnarray*}
(Q_0 ,Q_1 ,\dots ,Q_k )[(T_k \otimes I)(h_k^0 I,h_k^1 I,\dots, h_k^k
I)^T+  \lambda _k (h_k^0 I,h_k^1 I,\dots, h_k^k I)^T +(\gamma _0 I,0,\dots,
0)^T]=0.
\end{eqnarray*}
Hence\\[-1mm]
\begin{equation*}
f^*(f(\tilde {X}_k)) +\lambda _k \tilde {X}_k -f^*(E) -\gamma _{k+1} h_k
^kQ_{k+1} =0.
\end{equation*} 
\end{proof}

 \end{appendices}

\end{document}